\documentclass[11pt]{amsart}

\usepackage{amsmath,amssymb,amscd,amsthm,amsxtra}
\usepackage[dvips]{graphics,epsfig}

\allowdisplaybreaks[2]

\newtheorem{theorem}{Theorem} [section]

\newtheorem{lemma}[theorem]{Lemma}
\newtheorem{proposition}[theorem]{Proposition}

\newtheorem{corollary}[theorem]{Corollary}

\DeclareMathOperator*{\intt}{\int}

\DeclareMathOperator{\MAX}{MAX}

\newcommand{\noi}{\noindent}
\newcommand{\Z}{\mathbb{Z}}
\newcommand{\R}{\mathbb{R}}

\newcommand{\T}{\mathbb{T}}

\newcommand{\al}{\alpha}
\newcommand{\dl}{\delta}

\newcommand{\eps}{\varepsilon}
\newcommand{\kk}{\kappa}
\newcommand{\g}{\gamma}

\newcommand{\ld}{\lambda}

\newcommand{\s}{\sigma}
\newcommand{\ft}{\widehat}
\newcommand{\wt}{\widetilde}
\newcommand{\cj}{\overline}
\newcommand{\dx}{\partial_x}
\newcommand{\dt}{\partial_t}
\newcommand{\dd}{\partial_\dl}

\newcommand{\jb}[1]
{\langle #1 \rangle}

\begin{document}

\baselineskip = 20pt

\title
[Invariant Gibbs Measure for SBO System]
{\bf Invariance of the Gibbs Measure  for the Schr\"odinger-Benjamin-Ono System}

\author{Tadahiro Oh}

\address{Tadahiro Oh\\
Department of Mathematics\\
University of Toronto\\
40 St. George St, Rm 6290,
Toronto, ON M5S 2E4, Canada}

\email{oh@math.toronto.edu}



\begin{abstract}
We prove the invariance of the Gibbs measure for the periodic Schr\"odinger-Benjamin-Ono system
(when the coupling parameter $|\g| \ne 0, 1$) by establishing a new local well-posedness in a modified Sobolev space
and constructing the Gibbs measure (which is in the sub-$L^2$ setting for the Benjamin-Ono part.)
We also show the ill-posedness result in $H^s(\T) \times H^{s-\frac{1}{2}}(\T)$
for $s < \frac{1}{2}$ when $|\g| \ne 0,  1$ and for any $s \in \mathbb R $ when $|\g| =1$.
\end{abstract}

\maketitle

\section{Introduction}

In this paper, we consider the Schr\"odinger-Benjamin-Ono (SBO) system:
\begin{equation} \label{SBO}
\left\{
\begin{array}{l}
i u_t + u_{xx}  = \al vu \\
v_t + \g \mathcal{H}  v_{xx}  = \beta(|u|^2)_x \\
(u, v) \big|_{t = 0} = (u_0, v_0) 
\end{array}
\right. ,
\ \ \
(x, t) \in \T \times\R 
\end{equation}

\noindent
where $ \T = [0, 2\pi)$,  $u$ is a complex valued function,
 $v$ is a real-valued function, and $\al, \beta, \g$ are nonzero real constants.
In \eqref{SBO}, $\mathcal{H}$ denotes the Hilbert transform  whose Fourier multiplier is given by
$ -i \text{sgn}(n)$.
$D = |\dx| = \mathcal{H}\dx$ is defined via $\ft{Df}(n) = |n|\ft{f}(n)$.

The system \eqref{SBO} appears in Funakoshi-Oikawa \cite{FO},
describing the motion of two fluids with different densities under capillary-gravity waves in a deep water flow.
The Sch\"odinger part describes the short surface wave, 
and the Benjamin-Ono part describes the long internal wave.
The system also appears in the sonic-Langmuir wave interaction in plasma physic (Karpman \cite{Karpman}),
in the capillary gravity interaction waves (Djordjevic-Redekopp \cite{DR}, Grimshaw \cite{Grimshaw}),
and in the general theory of water wave interaction in a nonlinear medium (Benney \cite{Benney1, Benney2}.)

The several conservation laws are known for the SBO system:
\begin{align}  \label{CONSERVED1}
E_1 (v) = \int v  \, dx, \quad & E_2 (u) = \int |u|^2 dx, \quad
E_3 (u, v) = \text{Im} \int u \cj{u_x} dx + \frac{\al}{2\beta} \int v^2 dx \\
\label{CONSERVED2}
\text{and } \quad & H(u, v)  
= \frac{1}{2}\int |u_x|^2 dx - \frac{\al\g}{4\beta}\int(D^\frac{1}{2}v)^2dx  + \frac{\al}{2} \int v |u|^2 dx.
\end{align}

\noindent
Note that $H(u, v)$ is the Hamiltonian for \eqref{SBO} and indeed, in terms of the Hamiltonian formulation,
\eqref{SBO} can be written as
\begin{equation} 
\dt \begin{pmatrix}u\\v\\\end{pmatrix} = \begin{pmatrix} i & 0 \\0 & \frac{2\beta}{\al}\dx \end{pmatrix}\frac{d H(u, v)}{d (\cj{u}, v)},
\end{equation}

\noindent
where $\frac{dH}{d(\cj{u}, v)}$ is the Gateaux derivative with respect to the $L^2$ inner product
\[ \jb{(u_1, v_1), (u_2, v_2)} = \text{Re} \int u_1 \cj{u_2} dx + \int v_1 v_2 dx.\]

In this formulation, one natural questions is;
 Is the Gibbs measure of the form $d \mu = Z^{-1} e^{-\nu H(u, v)}\prod_{x \in \T} du(x) \otimes dv(x)$ invariant under the flow of \eqref{SBO}?  
At this point, $d \mu$ is merely a formal expression.
It is known (c.f. Zhidkov \cite{Z}) that when $\frac{\al\g}{\beta} < 0$, the Gaussian part of the Gibbs measure,
i.e. 
\begin{equation} \label{GAUSS}
d \rho = \wt{Z}^{-1}\exp{ \Big( -\nu \big(\frac{1}{2}\int |u_x|^2  - \frac{\al \g}{4\beta} \int (D^\frac{1}{2}v)^2 \big)\Big)}
\prod_{x \in \T} du(x) \otimes dv(x),
\end{equation}

\noindent
 is supported in $\cap_{s < \frac{1}{2}} H^s(\T) \times H^{s-\frac{1}{2}}(\T)$. 
Following Bourgain \cite{BO4}, we show that the Gibbs measure 
$d \mu $
is a well-defined probability measure 
with a suitable cutoff in terms of the $L^2$ norm of $u$ and the mean of $v$
(Lemma \ref{LEM:GABSCONTI}, Corollary \ref {tight}.)

Now, let's turn to the well-posedness theory of \eqref{SBO}.
In the non-periodic setting, several results are known.
When $|\g| \ne 1$, Bekiranov-Ogawa-Ponce \cite{BOP2} 
showed that \eqref{SBO} is locally well-posed in $H^s(\R) \times H^{s-\frac{1}{2}}(\R)$ for $s \geq 0$.
When $|\g | = 1$, Pecher \cite{PECH} showed the local well-posedness for $s > 0$. 
In view of the conservation laws, when $\frac{\al\g}{\beta} < 0$, these local results extend to the global ones for $s \geq 1$.
Using the $I$-method developed by Colliander-Keel- Staffilani-Takaoka-Tao \cite{CKSTT2,CKSTT4},
Pecher \cite{PECH} also proved the global well-posedness for $s > \frac{1}{3}$ when $\frac{\al\g}{\beta} < 0$.
Note that the $I$-method automatically provides a polynomial upper bound on the time growth of the norm of the solutions. 
Recently, when $|\g| \ne 1$, Angulo-Matheus-Pilod proved the global well-posedness for $ s= 0$ 
(without assuming $\frac{\al\g}{\beta} < 0$), 
following the method developed by Colliander-Holmer-Tzirakis \cite{CHT}.
The method is based on estimating the doubling time of $\|v(t)\|_{H_x^{-\frac{1}{2}}}$ in terms of the conserved quantity $\|u(t)\|_{L^2_x}$.
Note that this method provides an exponential upper bound on the time growth of $\|v(t)\|_{H_x^{-\frac{1}{2}}}$.

In the periodic setting, there seems to be only few results known at this point.
Assuming $|\g| \ne 0, 1$, Angulo-Matheus-Pilod \cite{AMP}
showed that \eqref{SBO} is locally well-posed in $H^s(\T) \times H^{s-\frac{1}{2}}(\T)$ for $s \geq \frac{1}{2}$.
They also showed the existence and stability of the periodic travelling wave solutions.
In \cite{AMP}, the local well-posedness is established via contraction argument
by establishing the following bilinear estimates:
\begin{align} \label{bilinear1}
\|uv\|_{X^{s, -\frac{1}{2}+} } & \lesssim \|u\|_{X^{s, \frac{1}{2}}} \|v\|_{X_\g^{s - \frac{1}{2}, \frac{1}{2}}} \\
 \label{bilinear2}
\|\dx(u_1\cj{u_2})\|_{X_\g^{s-\frac{1}{2}, -\frac{1}{2}+} } 
& \lesssim \|u_1\|_{X^{s, \frac{1}{2}}} \|u_2\|_{X^{s, \frac{1}{2}}} 
\end{align}

\noindent
for $s \geq \frac{1}{2}$
where $X^{s, b}$ and $X_\g^{s, b}$ are the Bourgain spaces corresponding to the linear parts of 
Schr\"odinger and Benjamin-Ono equations whose norms are given by
\begin{align} \label{XSB}
\|u\|_{X^{s, b}} & = \|\jb{n}^s \jb{\tau + n^2}^b \ft{u}(n, \tau) \|_{L^2_{n, \tau}}  \\
\label{XSBG}
\|v\|_{X_\g^{s, b}} & = \|\jb{n}^s \jb{\tau + \g |n|n}^b \ft{v}(n, \tau) \|_{L^2_{n, \tau}},
\end{align}

\noindent
where $\jb{\, \cdot \,} = 1 + |\cdot|$.
It is also shown in \cite{AMP} that both estimates \eqref{bilinear1} and \eqref{bilinear2} fail for $s < \frac{1}{2}$ when $|\g| \ne 0, 1$
and that they fail for any $s \in \R$ when $|\g| = 1$.

In Appendix, we show that the solution map of \eqref{SBO} is not smooth if $s < \frac{1}{2}$ for $|\g| \ne 0,  1$.
More precisely, let $\Phi^t : (u_0, v_0) \mapsto (u(t)), v(t)) \in H^s (\T) \times H^{s-\frac{1}{2}}(\T)$
be the solution map of \eqref{SBO} for $|t| \ll1$.
Then, 

\begin{theorem} \label{THM:illposed}

\textup{(a)} Let $|\g| \ne 0, 1$. If the solution map $\Phi^t$ is $C^2$ on $H^s (\T) \times H^{s-\frac{1}{2}}(\T)$, then $s \geq \frac{1}{2}.$
\textup{(b)} If $|\g| = 1$, then the solution map can never be $C^2$ for any $s\in \R$.
\end{theorem}

\noindent
In particular, Theorem \ref{THM:illposed} states 
that the flow of \eqref{SBO} via the usual contraction argument with the Bourgain norm
is {\it not} defined in $\cap_{s < \frac{1}{2}} H^s(\T) \times H^{s-\frac{1}{2}}(\T)$
containing the support of the Gibbs measure $\mu$.
Thus, we need to seek for a new local well-posedness result in a space containing the support of $\mu$.

From the standard argument (c.f. Kenig-Ponce-Vega \cite{KPV4}), 
the proof of the bilinear estimate
$\|uv\|_{X^{s, 1-b} }\lesssim \|u\|_{X^{s, b}} \|v\|_{X_\g^{s - \frac{1}{2}, b}}$ 
with $b = \frac{1}{2}$ or $\frac{1}{2}+$ comes down to establishing an effective upper bound
(which needs to be at most of the order 1 in view of Bourgain's periodic $L^4$ Strichartz estimate \cite{BO1}. 
See Lemma \ref{LEM:L4}) on
\begin{equation} \label{IN1}
 \frac{\jb{n}^s} {\jb{n_1}^s \jb{ n_2}^{s-\frac{1}{2}}} \frac{1}{ \max\big( \jb{\tau+ n^2}^{1-b},  \jb{\tau_1 + n_1^2}^b,  \jb{\tau_2 + \g |n_2|n_2}^b\big) },
\end{equation}

\noindent
where $n = n_1 + n_2$ and  $\tau = \tau_1 + \tau_2$ with $n, n_1, n_2 \in \Z$.
Note that
\begin{align} \label{IN2}
\max( \jb{\tau+ n^2},  & \jb{\tau_1 + n_1^2},  \jb{\tau_2 + \g |n_2|n_2}) 
\gtrsim | -(\tau+ n^2) + (\tau_1 + n_1^2) + (\tau_2  + \g |n_2|n_2 )| \notag \\
& = |n_2| | R_n(n_2) |,
\end{align}

\noindent
where $R_n(n_2) = (\g \, \text{sgn}(n_2) + 1)n_2 - 2n$.
For simplicity, assume $n, n_2 > 0$.
Then, when $n_2 \sim \frac{2n}{1+ \g}$, we have $R_n(n_2) \sim 0$.
In particular, if $|\g| \ne 1$, then we have $|n| \sim |n_1| \sim |n_2|$ in such a situation.
Then, we have
$ \eqref{IN1} \lesssim \frac{\jb{n}^s} {\jb{n_1}^s \jb{ n_2}^{s-\frac{1}{2}}} \sim \jb{n}^{-s+\frac{1}{2}}$,
which forces us to restrict ourselves to the case $s \geq \frac{1}{2}$.
However, note that for each $n \in \Z$, there are at most two values of $n_2$, i.e. $n_2 = \big[\frac{2n}{1+ \g}\big]$, 
$\big[\frac{2n}{1+ \g}\big] + 1$ which makes $R_n(n_2) \sim 0$. 
For all other values of $n_2$, we have $|R_n(n_2)| \gtrsim 1$.
Thus, with $b = \frac{1}{2}$, we have 
$ \eqref{IN1} \lesssim \frac{\jb{n}^s} 
{\jb{n_1}^s \jb{ n_2}^{s-\frac{1}{2}}|n_2|^\frac{1}{2}} \sim \frac{\jb{n}^s} {\jb{n_1}^s \jb{ n_2}^{s}}
\lesssim 1, $
on $A = \big\{ (n, n_1, n_2) : n = n_1 + n_2, \ \big|n_2 - \frac{2n}{1+\g}\big| > 1\big\}$
as long as $ s \geq 0$.

This motivates us to consider the initial value problem \eqref{SBO}
with the initial condition 
$(u_0, v_0) \in \text{H}^{s_1, s_2} (\T) := H^{s_1, s_2}(\mathbb{T}) \times H^{s_1-\frac{1}{2},\,  s_2-\frac{1}{2}} (\mathbb{T})$, 
where  $H^{s_1, s_2}$ is defined via the norm given by
\begin{equation} \label{GibbsIV} 
\| \phi \|_{H^{s_1, s_2}} = \| \phi \|_{H^{s_1}} + \sup_{n} \jb{n}^{s_2} |\ft{\phi}(n)| < \infty 
\end{equation}

\noindent
for some $s_1, s_2$ with $0 < s_1 = \frac{1}{2}- < \frac{1}{2} < s_2 = 1- < 1$ 
(with some additional conditions to be determined later.)

From the above heuristic argument,  we see that $s\geq 0$
is enough to establish the crucial bilinear estimates on $A$.
In particular, $s_1 = \frac{1}{2}-$ is a sufficient regularity on $A$.
On the other hand, the resonances at $n_2 = \big[\frac{2n}{1+ \g}\big]$, $\big[\frac{2n}{1+ \g}\big] + 1$
forces $s \geq \frac{1}{2}$.
However, for each fixed $n \in \Z$, there are only two values of $n_2$ causing the resonances, 
which can be controlled by 
$\sup_{n} \jb{n}^{s_2} |\ft{\phi}(n)|$ in \eqref{GibbsIV} 
with the higher regularity $s_2 = 1- > \frac{1}{2}$.
Then, via a contraction in the modified Bourgain space, we  prove

\begin{theorem} \label{THM:LWP}
Let $s _1 = \frac{1}{2}-$, $s_2 = 1-$ with $s_2 < 2s_1$.
Assume $|\g| \ne 0, 1$.
Then, the SBO system \eqref{SBO} is locally well-posed in $\textup{\text{H}}^{s_1, s_2} (\T)$.
\end{theorem}

Note that $\text{H}^{s_1, s_2} \subsetneq H^{s_1} \times H^{s_1-\frac{1}{2}}$.
However, from the theory of abstract Wiener spaces (Gross \cite{GROSS}, Kuo \cite{KUO}),
the Gaussian measure $d \rho$ in \eqref{GAUSS}
is a countably additive probability measure supported on 
$\text{H}^{s_1, s_2}$ for
$0 < s_1  < \frac{1}{2} < s_2  < 1$.  
(See Bourgain \cite{BO4} for mKdV and \cite{BO5} for the Zakharov system.)
Then, a slight modification Bourgain's argument \cite{BO4} (for super-cubic NLS) 
with Theorem \ref{THM:LWP} yields the following result.

\begin{theorem} \label{THM:INVARIANCE}
Let $\frac{\al \g}{\beta} < 0$.
The Gibbs measure $\mu$ for the SBO system \eqref{SBO}  is invariant under the flow.
Moreover, \eqref{SBO} is globally well-posed almost surely on the statistical ensemble.
\end{theorem}

\section{Notations}

On $\T$, the spatial Fourier domain is $\Z$.
Let $dn$ be the normalized counting measure on $\Z$, 
and we say $f \in L^p(\Z)$, $1 \leq p < \infty$, if
\[ \| f \|_{L^p(\mathbb{Z})} = \bigg( \int_{\mathbb{Z}} |f(n)|^p dn \bigg)^\frac{1}{p}  
:= \bigg( \frac{1}{2\pi} \sum_{n \in \mathbb{Z}} |f(n)|^p \bigg)^\frac{1}{p} < \infty.\]

\noindent
If $ p = \infty$, we have the obvious definition involving the essential supremum.
We often drop $2\pi$ for simplicity.
If the function depends on both $x$ and $t$, we use ${}^{\wedge_x}$ 
(and ${}^{\wedge_t}$) to denote the spatial (and temporal) Fourier transform, respectively.
However, when there is no confusion, we simply use ${}^\wedge$ to denote the spatial Fourier transform,
temporal Fourier transform, and  the space-time Fourier transform, depending on the context.

Let $X^{s, b}$ and $X_\g^{s, b}$ be as in \eqref{XSB} and \eqref{XSBG}.
Given any time interval $I = [t_1, t_2]\subset \mathbb{R}$, we define the local in time $X^{s, b}(\T \times I )$ 
(or simply $X^{s, b}[t_1, t_2]$) by
\[ \|u\|_{X_I^{s, b}} 
= \|u \|_{X^{s, b}(\T \times I )} = \inf \big\{ \|\wt{u} \|_{X^{s, b}(\T \times \mathbb{R})}: {\wt{u}|_I = u}\big\}.\]

\noindent
We define the local in time $X^{s, b}_\g (\T \times I )$ analogously.

Let $\eta \in C^\infty_c(\mathbb{R})$ be a smooth cutoff function supported on $[-2, 2]$ with $\eta \equiv 1$ on $[-1, 1]$
and let $\eta_{_T}(t) =\eta(T^{-1}t)$. 
We use $c,$ $ C$ to denote various constants, usually depending only on $s_1, s_2, b, \alpha, \beta$, and  $\g$.
If a constant depends on other quantities, we will make it explicit.
We use $A\lesssim B$ to denote an estimate of the form $A\leq CB$.
Similarly, we use $A\sim B$ to denote $A\lesssim B$ and $B\lesssim A$
and use $A\ll B$ when there is no general constant $C$ such that $B \leq CA$.
We also use $a+$ (and $a-$) to denote $a + \eps$ (and $a - \eps$), respectively,  
for arbitrarily small $\eps \ll 1$.

\section{New Local Well-Posedness on the Modified Sobolev Space}

In this section, we establish the local well-posedness of \eqref{SBO} 
with the initial data in $\text{H}^{s_1, s_2}
 = H^{s_1, s_2} \times H^{s_1-\frac{1}{2}, s_2-\frac{1}{2}}$ for $s_1 = \frac{1}{2} -$ and $s_2 = 1-$ 
via a contraction on a ball 
in the modified Bourgain space $X^{s_1, s_2, b} \times X_\g^{s_1-\frac{1}{2}, s_2-\frac{1}{2}, b}$.
First, define the modified Bourgain spaces $X^{s_1, s_2, b}$ and $X_\g^{s_1, s_2, b}$ whose norms are given by
\begin{align} \label{NEWXSB}
\| u \|_{X^{s_1, s_2, b}} & = \|u\|_{X^{s_1, b}} + \|u\|_{X^{s_2, \infty, b}} \\
\label{NEWXSBG}
\| v \|_{X_\g^{s_1, s_2, b}} & = \|v\|_{X_\g^{s_1, b}} + \|v\|_{X_\g^{s_2, \infty,  b}}, 
\end{align}

\noindent
where $X^{s_1, b}$ and $X_\g^{s_1, b}$ are defined in \eqref{XSB} and \eqref{XSBG},
and 
\begin{align}
\|u\|_{X^{s_2, \infty,  b}}  & = \| \jb{n}^{s_2} \jb{\tau + n^2}^b \ft{u}(n, \tau) \|_{L^\infty_n L^2_\tau} \\
 \|v\|_{X_\g^{s_2, \infty,  b}} & =  \| \jb{n}^{s_2} \jb{\tau + \g|n|n}^b \ft{v}(n, \tau) \|_{L^\infty_n L^2_\tau}.
 \end{align} 

Recall that  when $b > \frac{1}{2}$, the $X^{s_1, b} \times X_\g^{s_1-\frac{1}{2}, b}$ norm controls the $C([-T, T] ; H^{s_1} \times H^{s_1-\frac{1}{2}})$ norm.
Also, when $b > \frac{1}{2}$, by Sobolev embedding, we have
\[\sup_n \jb{n}^{s_2}  | \ft{u}(n, t)| 
= \sup_n \jb{n}^{s_2} | e^{in^2t}\ft{u}(n, t)|
 \lesssim \sup_n \jb{n}^{s_2} \| e^{in^2t}\ft{u}(n, t)\|_{H^b_t}   
 = \| u \|_{X^{s_1, s_2, b}},\]

\noindent
 for any  $t \in \mathbb{R}$. A similar result holds if we replace $\jb{\tau + n^2}^b$ by $\jb{\tau + \g|n|n}^b$.
Hence, the $X^{s_1, s_2, b} \times X_\g^{s_1-\frac{1}{2}, s_2-\frac{1}{2}, b} $ norm controls 
the $C([-T, T] ; \text{H}^{s_1, s_2} )$ norm.

By writing \eqref{SBO} in the integral form, we see that $(u, v)$ is a solution to \eqref{SBO}
with the initial condition $(u_0, v_0)$ for $|t| \leq T \leq 1$ if and only if
\begin{equation} 
\begin{pmatrix} 
u(t) \\ v(t) 
\end{pmatrix} 
= \Phi^t_{(u_0, v_0)} (u, v)
:=
\begin{pmatrix} 
\eta (t) U(t) u_0 + i \al \eta_{_T}(t)  \int_0^t U(t - t') uv (t') dt' \\
\eta (t) V(t) v_0 - \beta \eta_{_T}(t) \int_0^t V(t - t') \dx \big( |u(t')|^2\big)  dt' \\
\end{pmatrix},
\end{equation}

\noindent
where $U(t) = e^{it\dx^2}$ and $V(t) = e^{-\g t \mathcal{H} \dx^2}$.
First, note that 
$(\eta(t)U(t) u_0)^{\wedge}(n, \tau) = \ft{\eta}(\tau +n^2) \ft{u_0}(n)$
and $(\eta(t)V(t) v_0)^{\wedge}(n, \tau) = \ft{\eta}(\tau +\g |n|n) \ft{v_0}(n)$.
Hence, we have
\begin{align} \label{LINEAR1}
 \|  \eta(t) U(t) u_0  \|_{X^{s_1, s_2, b}}  & \lesssim \| u_0  \|_{H^{s_1, s_2}}
 \\ \label{LINEAR2}
\| \eta(t) V(t) v_0  \|_{X_\g^{s_1-\frac{1}{2}, s_2-\frac{1}{2}, b}} & \lesssim \| v_0  \|_{H^{s_1-\frac{1}{2}, s_2-\frac{1}{2}}}.
\end{align}

\noindent
Now, let $-\frac{1}{2} < b' \leq 0 \leq b \leq b'+1$ and $T\leq 1$. 
Then, from (2.25) in Lemma 2.1 (ii) in Ginibre-Tsutsumi-Velo \cite{GTV}, we have
\begin{align} \label{DUHAMEL1}
  \| \eta_{_T} (U*_RF)\|_{X^{s_1, b}} & \lesssim T^{1-b + b'}\| F\|_{X^{s_1, b'}}
\\ \label{DUHAMEL2}  
\| \eta_{_T} (V*_R G)\|_{X_\g^{s_1 -\frac{1}{2}, b}} & \lesssim T^{1-b + b'}  \|G\|_{X_\g^{s_1-\frac{1}{2}, b'}}, 
\end{align}

\noindent
where $*_R$ denotes the retarded convolution, i.e. $ U*_R F(t) = \int_0^t U(t - t')  F (t')  dt'$.
Also, from (2.24) in Lemma 2.1 (ii) in \cite{GTV}, we have
$\|\eta_{_T} (U*_RF)\|_{H^b_t} 
 \lesssim T^{1-b + b'} \|   U( - t)    F (t)  \|_{H^b_t}.$
Noting that $\|u\|_{X^{s_2, \infty, b}} = \big\|\big(U(-t) u\big)^{\wedge_x}(n, t)\big\|_{L^\infty_n H^b_t}$,
we have 
\begin{equation} \label{DUHAMEL3}
\| \eta_{_T} (U*_R F)\|_{X^{s_2, \infty, b}}  \lesssim T^{1-b+b'} \|  F\|_{X^{s_2, \infty, b'}}.
\end{equation}

\noindent
The same computation holds if we replace $U(t)$ by $V(t)$, and thus we obtain
\begin{equation} \label{DUHAMEL4}
\| \eta_{_T} ( V*_RG)\|_{X_\g^{s_2, \infty, b}}  \lesssim T^{1-b+b'} \|  G \|_{X_\g^{s_2, \infty, b'}}.
\end{equation}

\noindent
Then, the local well-posedness of \eqref{SBO} in $\text{H}^{s_1, s_2}$ 
follows from the standard argument \cite{BO1, KPV4}
once we prove the following bilinear estimates.

\begin{proposition} \label{PROP:Gibbsbilinear1}
Let $b = \frac{1}{2}+$ and $b' = -\frac{1}{2}+$ with $ 1+ b' > b$.
Then, for $0 < s_1 < \frac{1}{2} \leq s_2 < 1$ with $s_2 < 2s_1$, we have
\begin{equation} \label{Gibbsbilinear1}
\| uv \|_{X^{s_1, s_2, b'}} \lesssim \|u \|_{X^{s_1, s_2, b}} \|v \|_{X_\g^{s_1-\frac{1}{2}, s_2-\frac{1}{2}, b} }.
\end{equation}
\end{proposition}

\begin{proposition}  \label{PROP:Gibbsbilinear2}
Let $b, b', s_1, s_2$ be as in Proposition \ref{PROP:Gibbsbilinear1}.
Then, we have
\begin{equation} \label{Gibbsbilinear2}
\| \dx( u_1 \cj{u_2}) \|_{X_\g^{s_1-\frac{1}{2}, s_2-\frac{1}{2}, b'}} \lesssim \| u_1 \|_{X^{s_1, s_2, b}} \| u_2 \|_{X^{s_1, s_2, b}}.
\end{equation}
\end{proposition}

More precisely, for given $\theta = 0+$, choose $b = \frac{1}{2} + \theta$ 
and $b'  -\frac{1}{2} + 2\theta$.
We prove \eqref{Gibbsbilinear1} for $s_1 \geq 2 \theta$, $s_2 \geq \frac{1}{2}$,  
$s_2 \leq \min ( 2 s_1 - 4\theta, s_1 + \frac{1}{2} -2\theta)$, and
 \eqref{Gibbsbilinear2} for $s_1 \geq 2 \theta$, $s_2 \geq \frac{1}{2}$, 
$s_2 \leq \min ( 2 s_1 - 2\theta, s_1 + \frac{1}{2} -4\theta)$.
Note that for $s_1 < \frac{1}{2}$, it is enough to take $s_2 \leq 2 s_1 - 4\theta$ in both cases.

Note that we have $ 1- b + b' = \theta > 0$.
Thus, the linear estimates \eqref{LINEAR1} $\sim$ \eqref{DUHAMEL4}
yield a small positive power of $T$ and 
establish the contraction property of $\Phi^t_{(u_0, v_0)} (\cdot, \cdot)$ for $|t| \leq T \ll 1$.
As a

Before proving of Propositions \ref{PROP:Gibbsbilinear1} and \ref{PROP:Gibbsbilinear2}, 
first recall the $L^4$ Strichartz estimate due to Bourgain \cite{BO1}.
Also, see Molinet \cite{MO}.

\begin{lemma} \label{LEM:L4}
Let $\g \ne 0$. Then, we have
\[ \|u\|_{L^4_{x, t} }\lesssim \|u\|_{X^{0, \frac{3}{8}}},  \text { and } \ \|v\|_{L^4_{x, t} }\lesssim \|v\|_{X_\g^{0, \frac{3}{8}}}.\]
\end{lemma}

\begin{proof} [Proof of Proposition \ref{PROP:Gibbsbilinear1}]  
In the proof, we use $(n, \tau)$, $(n_1, \tau_1)$, $(n_2, \tau_2)$ to denote the Fourier variables 
for $uv$, $u$, $v$, respectively.
i.e. we have $n = n_1 + n_2$ and $\tau = \tau_1 + \tau_2$ with $n, n_1, n_2 \in \Z$ and $\tau, \tau_1, \tau_2 \in \R$.
Also, let $\s = \jb{\tau+ n^2}$, $\s_1 = \jb{\tau_1 + n_1^2}$, and $\s_2 = \jb{\tau_2 + \g|n_2|n_2}$.
Then, as in \eqref{IN2}, we have  
\begin{equation} 
\MAX := \max( \s,  \s_1,  \s_2) 
 \gtrsim 1 + |n_2| | R_n(n_2) |,
\end{equation}

\noindent
where $R_n(n_2) = (\g \, \text{sgn}(n_2) + 1)n_2 - 2n$.
Note that for fixed $n\in \Z$, we have  $R_n(n_2) = 0$ when $n_2 = \frac{2n}{1 + \g \text{sgn}(n_2)}$, 
(which may not be an integer.)
Hence, we have $|R_n(n_2)| \gtrsim 1$ for 
$n_2 \ne \Big[\frac{2n}{1 + \g \text{sgn}(n_2)}\Big], \Big[\frac{2n}{1 + \g \text{sgn}(n_2)}\Big]+ 1$.

\noindent
$\bullet$ {\bf Part 1:} 
We first prove \eqref{Gibbsbilinear1} for the $X^{s_1, b'}$ part of the $X^{s_1, s_2, b'}$ norm.

Define the bilinear operator $\mathcal{B}_{s_1, \theta} (\cdot, \cdot)$ by
\begin{equation} \label{B1}
\mathcal{B}_{s_1, \theta} (f, g) (n, \tau) = \frac{1}{2\pi} \sum_{n = n_1+ n_2} \intt_{\tau = \tau_1+ \tau_2}
\frac{\jb{n}^{s_1}}{\jb{n_1}^{s_1}\jb{n_2}^{s_1-\frac{1}{2}}} \frac{f(n_1 , \tau_1) g(n_2, \tau_2)}
{\s^{\frac{1}{2}-2\theta} \s_1^{\frac{1}{2}+\theta} \s_2^{\frac{1}{2}+\theta}} d\tau_1.
\end{equation}

\noindent
Then, it suffices to prove 
\begin{equation} \label{DUAL1}
\|\mathcal{B}_{s_1, \theta} (f, g)\|_{L^2_{n, \tau}} \lesssim \|f\|_{L^2_{n, \tau}}\|g\|_{L^2_{n, \tau}}.
\end{equation}

Let $A = \big\{ (n, n_1, n_2) : n = n_1 + n_2, \ \big|n_2 - \frac{2n}{1+\g \text{sgn}(n_2) }\big| \geq 1\big\}$
Then, on $A$, we have $\MAX \gtrsim \jb{n_2}$.
Hence, if $|n_1| \gtrsim |n|$, then we have 
\begin{equation} \label{NONRES1}
\frac{\jb{n}^{s_1}}{\jb{n_1}^{s_1}\jb{n_2}^{s_1-\frac{1}{2}}}
\frac{1}{\s^{\frac{1}{2}-2\theta} \s_1^{\frac{1}{2}+\theta} \s_2^{\frac{1}{2}+\theta}}
\leq \frac{\jb{n}^{s_1}}{\jb{n_1}^{s_1}\jb{n_2}^{s_1-\frac{1}{2}} \MAX^{\frac{1}{2} -2\theta}}
\lesssim \frac{\jb{n}^{s_1}}{\jb{n_1}^{s_1}\jb{n_2}^{s_1-2\theta}}
\lesssim 1
\end{equation}

\noindent
for $s_1 \geq 2 \theta$.
Now, suppose $|n_1| \ll |n|$. Then, we have $|n_2| \sim |n|.$
Moreover, we have
\begin{equation} \label{RN1}
|R_n(n_2)| = |(1+\wt{\g})n_2 - 2n| \geq |2-(1+\wt{\g})| |n| - |(1+\wt{\g}) n_1| \sim |n|\sim |n_2|.
\end{equation}

\noindent
where 
$\wt{\g} =  \g \, \text{sgn}(n_2) $.
Thus, $\MAX \gtrsim \jb{n_2}^2$ in this case, and we have
\begin{equation} \label{NONRES2}
\frac{\jb{n}^{s_1}}{\jb{n_1}^{s_1}\jb{n_2}^{s_1-\frac{1}{2}}}
\frac{1}{\s^{\frac{1}{2}-2\theta} \s_1^{\frac{1}{2}+\theta} \s_2^{\frac{1}{2}+\theta}}
\lesssim \frac{\jb{n}^{s_1}}{\jb{n_1}^{s_1}\jb{n_2}^{s_1 + \frac{1}{2}-4\theta}}
\lesssim 1
\end{equation}

\noindent
for $s_1 \geq 0$ (with $\theta$ sufficiently small. i.e. $\theta \leq \frac{1}{8}$.)

If $\MAX = \s$, then by H\"older inequality and Lemma \ref{LEM:L4}, we have 
\[\|\mathcal{B}_{s_1, \theta} (f, g)\|_{L^2_{n, \tau}} 
\leq \big\| \big(\s_1^{ - \frac{1}{2}-\theta} f \big)^{\vee}\big\|_{L^4_{x, t}}
\big\| \big(\s_2^{- \frac{1}{2}-\theta} g \big)^{\vee}\big\|_{L^4_{x, t}} 
\lesssim \|f\|_{L^2_{n, \tau}}\|g\|_{L^2_{n, \tau}}\]

\noindent
on $A$.
If $\MAX = \s_1$, then  by H\"older inequality and Lemma \ref{LEM:L4}, we have, 
for any $h \in L^2_{n, \tau}$, 
\[ \jb{\mathcal{B}_{s_1, \theta} (f, g), h}_{L^2_{n, \tau} }
\leq \|f\|_{L^2_{n, \tau}}  \big\| \big(\s_2^{ - \frac{1}{2}-\theta} g \big)^{\vee}\big\|_{L^4_{x, t}}
\big\| \big(\s^{ - \frac{1}{2} + 2\theta} h \big)^{\vee}\big\|_{L^4_{x, t}}
 \lesssim \|f\|_{L^2_{n, \tau}}\|g\|_{L^2_{n, \tau}} \|h\|_{L^2_{n, \tau}}\]

\noindent
on $A$ and this is equivalent to \eqref{DUAL1} via duality.
A similar computation holds when $\MAX = \s_2$.
Hence, \eqref{DUAL1} holds on $A$ for $s_1 \geq 2 \theta$.

Now, we will consider the estimate on $A^c$, i.e. near the resonances.
\noindent
In this case, we show 
\begin{equation} \label{DUAL2}
\big\|\wt{\mathcal{B}}_{s_1, s_2, \theta} (f, g)\big\|_{L^2_{n, \tau}} \lesssim \|f\|_{L^2_{n, \tau}}\|g\|_{L^\infty_nL^2_{ \tau}},
\end{equation}

\noindent
where $\wt{\mathcal{B}}_{s_1, s_2, \theta}(\cdot, \cdot)$ is given by
\begin{equation} \label{B2}
\wt{\mathcal{B}}_{s_1, s_2, \theta}(f, g) (n, \tau) = \frac{1}{2\pi} \sum_{n = n_1+ n_2} \intt_{\tau = \tau_1+ \tau_2}
\frac{\jb{n}^{s_1}}{\jb{n_1}^{s_1}\jb{n_2}^{s_2-\frac{1}{2}}} \frac{f(n_1 , \tau_1) g(n_2, \tau_2)}
{\s^{\frac{1}{2}-2\theta} \s_1^{\frac{1}{2}+\theta} \s_2^{\frac{1}{2}+\theta}} d\tau_1.
\end{equation}

\noindent
First, note that since $|\g| \ne 1$, it follows that $|n| \sim |n_1| \sim |n_2|$ on $A^c$.
Thus, we have
$ \frac{\jb{n}^{s_1}}{\jb{n_1}^{s_1}\jb{n_2}^{s_2-\frac{1}{2}}} \sim \jb{n}^{\frac{1}{2} - s_2} \lesssim 1$
for $s_2 \geq \frac{1}{2}$.
On $A^c$, we can not expect any contribution from $\s$, $\s_1$, $\s_2$.
However, for fixed $n$, there are only finitely many (2 or 4) values of $n_2$ in $A^c$.
i.e. there is virtually no sum over $n_2$ in this case.
Let $F(n_1, \tau_1 ) = \s_1^{-\frac{1}{2}-\theta} f(n_1, \tau_1) $ and $G(n_2, \tau_2 ) = \s_2^{-\frac{1}{2}-\theta} g(n_2, \tau_2) $.
Then,  by H\"older inequality in $t$ and Sobolev inequality, we have
\begin{align}  \label{HOLDER1}
\bigg\|\intt_{\tau = \tau_1 + \tau_2} & F(n_1, \tau_1)  G(n_2, \tau_2) d\tau_1\bigg\|_{L^2_{ \tau}} 
\leq  \big\| \big(F(n_1, \cdot) \big)^{\vee_t} \big\|_{L^4_t} \big\| \big(G(n_2, \cdot) \big)^{\vee_t} \big\|_{L^4_t} \notag \\
& = \big\| e^{i n_1^2 t} \big(F(n_1, \cdot) \big)^{\vee_t} (t) \big\|_{L^4_t} \big\| e^{i \g |n_2|n_2 t} \big(G(n_2, \cdot) \big)^{\vee_t} (t) \big\|_{L^4_t}  \\
& \lesssim \big\| U(-t) \big(F(n_1, \cdot) \big)^{\vee_t} \big\|_{H^\frac{1}{4}_t}
\big\| V(-t) \big(G(n_2, \cdot) \big)^{\vee_t} \big\|_{H^\frac{1}{4}_t}. \notag 
\end{align}

\noindent
Then, for $n_2 =  \Big[\frac{2n}{1 + \g \text{sgn}(n_2)}\Big]$ or $ \Big[\frac{2n}{1 + \g \text{sgn}(n_2)}\Big]+ 1$,
we have,
\begin{align*} 
\text{LHS  of }&  \eqref{DUAL2} 
\lesssim \bigg\|\intt_{\tau = \tau_1 + \tau_2} F(n_1, \tau_1) G(n_2, \tau_2) d\tau_1\bigg\|_{L^2_{n, \tau}} \\
&\lesssim \| \jb{\tau + n_1^2}^\frac{1}{4} F(n_1, \tau) \|_{L^2_{n_1, \tau}} \| \jb{\tau + \g|n_2| n_2}^\frac{1}{4} G(n_2, \tau) \|_{L^\infty_{n_2} L^2_\tau} 
\leq \| f\|_{L^2_{n, \tau}} \|g\|_{L^\infty_nL^2_\tau}.
\end{align*}

\noindent
$\bullet$ {\bf Part 2:} Next, we prove \eqref{Gibbsbilinear1} for the $X^{s_2, \infty, b'}$ part of the $X^{s_1, s_2, b'}$ norm.

For $|n|\lesssim 1$, we have $L^\infty_n$ norm $\sim$ $L^2_n$ norm
and $\jb{n}^{s_2} \sim \jb{n}^{s_1}$, 
i.e. the proof reduces to Part 1.
Hence, we assume $|n| \gtrsim 1$.

On $A$, we have $\MAX \gtrsim \jb{n_2}$. 
Thus, if $|n_1|, |n_2| \gtrsim |n|$, then we have
\begin{equation} \label{NONRES3}
\frac{\jb{n}^{s_2}}{\jb{n_1}^{s_1}\jb{n_2}^{s_1-\frac{1}{2}}}
\frac{1}{\s^{\frac{1}{2}-2\theta} \s_1^{\frac{1}{2}+\theta} \s_2^{\frac{1}{2}+\theta}}
\lesssim \frac{\jb{n}^{s_2}}{\jb{n_1}^{s_1}\jb{n_2}^{s_1-2\theta}}
\lesssim 1
\end{equation}

\noindent
for $2 s_1 \geq s_2 + 2\theta$.
On the other hand, if $|n_1| \ll |n|$, we have $|n_2| \sim |n|$ and 
$|R_n(n_2)|\sim |n|$
as in \eqref{RN1}.
i.e. $\MAX \sim \jb{n}^2$.
Then, we have 
$\text{LHS of } \eqref{NONRES3} 
\lesssim \frac{\jb{n}^{s_2}}{\jb{n_1}^{s_1}\jb{n}^{2s_1-4\theta}}
\lesssim 1$ 
for $2 s_1 \geq s_2 + 4\theta$.
Also, if $|n_2| \ll |n|$, then  we have $|n_1| \sim |n|$ and 
$|R_n(n_2)|\sim |n|$.
i.e. $\MAX \sim \jb{n}\jb{n_2}$.
Then, we have
$\text{LHS of } \eqref{NONRES3} 
\lesssim \frac{\jb{n}^{s_2}}{\jb{n_1}^{s_1}\jb{n_2}^{s_1-2\theta}\jb{n}^{\frac{1}{2}-2\theta}}
\lesssim 1$ 
for $s_1 \geq 2\theta$ and $s_1 + \frac{1}{2} \geq s_2 + 2\theta$.

In this case, it suffices to show 
\begin{equation} \label{DUAL3}
\big\|\mathcal{B}'_{s_1, s_2, \theta} (f, g)\big\|_{L^{\infty}_n L^2_\tau} \lesssim \|f\|_{L^2_{n, \tau}}\|g\|_{L^2_{n,  \tau}},
\end{equation}

\noindent
where $\mathcal{B}_{s_1, s_2, \theta}(\cdot, \cdot)$ is given by
\begin{equation} \label{B3}
\mathcal{B}'_{s_1, s_2, \theta}(f, g) (n, \tau) = \frac{1}{2\pi} \sum_{n = n_1+ n_2} \intt_{\tau = \tau_1+ \tau_2}
\frac{\jb{n}^{s_2}}{\jb{n_1}^{s_1}\jb{n_2}^{s_1-\frac{1}{2}}} \frac{f(n_1 , \tau_1) g(n_2, \tau_2)}
{\s^{\frac{1}{2}-2\theta} \s_1^{\frac{1}{2}+\theta} \s_2^{\frac{1}{2}+\theta}} d\tau_1.
\end{equation}

\noindent
Suppose $\MAX = \s$. Then, from \eqref{HOLDER1} and Young's inequality in $n$, we have
\begin{align*}
\text{LHS of } & \eqref{DUAL3}
\lesssim \bigg\|\sum_{n = n_1 + n_2}\intt_{\tau = \tau_1 + \tau_2} F(n_1, \tau_1) G(n_2, \tau_2) d\tau_1\bigg\|_{L^\infty_nL^2_{ \tau}} \\
&\lesssim \| \jb{\tau + n_1^2}^\frac{1}{4} F(n_1, \tau) \|_{L^2_{n_1, \tau}} \| \jb{\tau + \g|n_2| n_2}^\frac{1}{4} G(n_2, \tau) \|_{L^2_{n_2, \tau}} 
\leq \| f\|_{L^2_{n, \tau}} \|g\|_{L^2_{n, \tau}},
\end{align*}

\noindent
where $F$ and $G$ are as in Part 1.
If $\MAX = \s_1$, then by H\"older inequality in $t$, Young's inequality in $n$,
 and Sobolev inequality (as in \eqref{HOLDER1}), we have
\begin{align*} \jb{ \mathcal{B}'& _{s_1,  s_2, \theta}  (f, g) , h}_{L^2_{n, \tau}}
\lesssim \big\langle f \ast (\s_2^{ - \frac{1}{2}-\theta}g), \,
\s^{- \frac{1}{2} + 2\theta} h \big\rangle_{L^2_{n, \tau}} \\
& \leq \|f\|_{L^2_{n, \tau}} \big\|(\s_2^{ - \frac{1}{2}-\theta}g)^{\vee_t}\|_{L^2_n L^4_t}
\big\|(\s^{ - \frac{1}{2}+2\theta}h)^{\vee_t}\|_{L^1_n L^4_t}
\lesssim \|f\|_{L^2_{n, \tau}}\|g\|_{L^2_{n, \tau}}\|h\|_{L^1_n L^2_\tau}
\end{align*}

\noindent
for any $h \in L^1_n L^2_\tau$,
and this is equivalent to \eqref{DUAL3} via duality.
A similar computation holds when $\MAX = \s_2$.
Hence, \eqref{DUAL3} holds on $A$ for $s_1 \geq 2 \theta$ 
and $s_2 \leq \min( 2 s_1 - 4 \theta, s_1 + \frac{1}{2} - 2\theta)$.

On $A^c$, we have $|n|\sim |n_1| \sim|n_2|$.
Thus, $\frac{\jb{n}^{s_2}}{\jb{n_1}^{s_2}\jb{n_2}^{s_2 -\frac{1}{2}}} \lesssim 1$ for $s_2 \geq \frac{1}{2}$.
Also, recall that for fixed $n$, there are only finitely many (2 or 4) values values of $n_2$ on $A^c$.
i.e. there is no sum over $n_2$ in this case. 
Then, as in \eqref{HOLDER1}, we have 
\begin{align*}
\bigg\| \frac{1}{2\pi} & \sum_{n = n_1+ n_2} \intt_{\tau = \tau_1+ \tau_2}
\frac{\jb{n}^{s_2}}{\jb{n_1}^{s_2}\jb{n_2}^{s_2-\frac{1}{2}}} \frac{f(n_1 , \tau_1) g(n_2, \tau_2)}
{\s^{\frac{1}{2}-2\theta} \s_1^{\frac{1}{2}+\theta} \s_2^{\frac{1}{2}+\theta}} d\tau_1 \bigg\|_{L^\infty_n L^2_\tau} \\
&\lesssim \bigg\| \intt_{\tau = \tau_1 + \tau_2} F(n_1, \tau_1) G(n_2, \tau_2) d \tau_1 \bigg\|_{L^\infty_n L^2_\tau} \\
&\lesssim \| \jb{\tau + n_1^2}^\frac{1}{4} F(n_1, \tau) \|_{L^\infty_{n_1} L^2_\tau} \| \jb{\tau + \g|n_2| n_2}^\frac{1}{4} G(n_2, \tau) \|_{L^\infty_{n_2} L^2_\tau} 
\leq \| f\|_{L^\infty_n L^2_\tau} \|g\|_{L^\infty_nL^2_\tau},
\end{align*}

\noindent
where $F$ and $G$ are as in Part 1.
This completes the proof of Proposition \ref{PROP:Gibbsbilinear1}.
\end{proof}

\begin{proof} [Proof of Proposition \ref{PROP:Gibbsbilinear2}]

Noting that $|n| \jb{n}^{r_1-\frac{1}{2}} \leq \jb{n}^{r_1+\frac{1}{2}}$, 
our goal is to show the boundedness of the multilinear functional $\mathcal{I}_{r_1, r_2, r_3, \theta}$ given by
\begin{align*}
\mathcal{I}_{r_1, r_2, r_3, \theta} = \frac{1}{4\pi^2} \sum_{\substack{n, n_1 \\ n = n_1+ n_2}} \intt_{\tau = \tau_1 + \tau_2} 
M_{r_1, r_2, r_3, \theta} \, f(n_1, \tau_1) g(n_2, \tau_2)h(n. \tau)   d\tau_1 d \tau,
\end{align*}

\noindent
for $h \in L^2_{n, \tau}$ or $L^1_n L^2_\tau$ 
and $r_j = s_1$ or $s_2$ with  $j = 1, 2, 3$, 
where 
\[M_{r_1, r_2, r_3, \theta} = \frac{\jb{n}^{r_1+\frac{1}{2}}} {\jb{n_1}^{r_2} \jb{n_2}^{r_3}} 
\frac{1} {\wt{\s}^{\frac{1}{2}-2\theta}\wt{\s}_1^{\frac{1}{2}+\theta}\wt{\s}_2^{\frac{1}{2}+\theta}}\]

\noindent
with
$\wt{\s} = \jb{\tau + \g |n|n}$, $\wt{\s}_1 = \jb{\tau_1 + n_1^2}$, and $\wt{\s}_2 = \jb{\tau_2 - n_2^2}$.
From a computation analogous to  \eqref{IN2}, we have
\[ \MAX := \max( \wt{\s}, \wt{\s}_1, \wt{\s}_2) \gtrsim 1+ |n| |\wt{R}_{n_1}(n)|, \]

\noindent
where 
$\wt{R}_{n_1}(n) = (\g \text{sgn}(n) + 1) n - 2n_1$.
Note that for fixed $n_1 \in \Z$, we have $\wt{R}_{n_1}(n) = 0$
when $n = \frac{2 n_1}{ 1+ \g \text{sgn}(n)}$.
Thus, we have  $|\wt{R}_{n_1}(n)| \gtrsim 1$ 
for $n \ne \Big[\frac{2 n_1}{ 1+ \g \text{sgn}(n)}\Big], \Big[ \frac{2 n_1}{ 1+ \g \text{sgn}(n)}\Big] + 1$.
We indicate how this proposition follows as a corollary to (the proof of) Proposition \ref{PROP:Gibbsbilinear1},
basically by replacing  $(n, \tau)$, $(n_1, \tau_1)$, $(n_2, \tau_2) $ here 
with $(n_2, \tau_2)$, $(n, \tau)$, $(-n_1, -\tau_1)$. 

First, consider the $X_\g^{s_1 -\frac{1}{2},\, b'}$ part of \eqref{Gibbsbilinear2}.
Let $B = \big\{ (n, n_1, n_2) :  n = n_1 + n_2, \ \big|n -\frac{2 n_1}{ 1+ \g \text{sgn}(n)}\big| \geq 1\big\}$.
On $B$, we have $\MAX \gtrsim \jb{n}$.
Thus, if $|n_1|, |n_2| \gtrsim |n|$, then, we have
$M_{s_1,s_1, s_1, \theta} \lesssim \jb{n}^{-s_1 + 2\theta } \lesssim 1$
for $s_1 \geq 2 \theta$.
If $|n_1| \ll |n|$, then we have $|n_2| \sim |n|$ and $\MAX \gtrsim \jb{n}^2$.
Also, if $|n_2| \ll |n|$, then we have $|n_1| \sim |n|$ and $\MAX \gtrsim \jb{n}^2$.
(See \eqref{RN1}.)
In both cases, we have $M_{s_1,s_1, s_1, \theta} \lesssim \jb{n}^{-\frac{1}{2} + 4\theta } \lesssim 1$.
Hence, by repeating the first half of Part 1 in the proof of Proposition \ref{PROP:Gibbsbilinear1},
we have 
$ \big|\mathcal{I}_{s_1, s_1, s_1, \theta}\big|
\lesssim \|f\|_{L^2_{n, \tau}}\|g\|_{L^2_{n, \tau}}\|h\|_{L^2_{n, \tau}},$
for all $h \in L^2_{n, \tau}$.

On $B^c$, we do not expect any contribution from $\wt{\s}, \wt{\s}_1, \wt{\s}_2$.
However, we have $|n|\sim|n_1|\sim|n_2|$.
Thus, $\frac{\jb{n}^{s_1+\frac{1}{2}}}{\jb{n_1}^{s_1} \jb{n_2}^{s_2}} \lesssim \jb{n}^{\frac{1}{2} -s_2} \lesssim 1$
for $s_2 \geq \frac{1}{2}$.
Also, note that for fixed $n_1$, there are only finitely many values of $n$ on $B^c$.
Hence, by repeating the second half of Part 1 in the proof of Proposition \ref{PROP:Gibbsbilinear1},
we have 
$ \big|\mathcal{I}_{s_1, s_1, s_2, \theta}\big|
\lesssim \|f\|_{L^2_{n, \tau}}\|g\|_{L^\infty_n L^2_\tau}\|h\|_{L^2_{n, \tau}}$,
for all $h \in L^2_{n, \tau}$.
This proves the $X_\g^{s_1 -\frac{1}{2},\, b'}$ part of \eqref{Gibbsbilinear2}.

Next, consider the $X_\g^{s_2 -\frac{1}{2}, \, \infty, \, b'}$ part of \eqref{Gibbsbilinear2}.
On $B$, we have $\MAX \gtrsim \jb{n}$.
If $|n_1|, |n_2| \gtrsim |n|$, then we have
$M_{s_2,s_1, s_1, \theta} \lesssim \jb{n}^{s_2-2s_1 + 2\theta } \lesssim 1$
for $2 s_1 \geq s_2 + 2 \theta$.
If $|n_1| \ll |n|$ or  $|n_2| \ll |n|$, then we have $\MAX \sim \jb{n}^2$
and 
$M_{s_2,s_1, s_1, \theta} \lesssim \jb{n}^{s_2 - s_1 - \frac{1}{2} + 4\theta } \lesssim 1$
for $s_1 + \frac{1}{2} \geq s_2 + 4\theta$.
Hence, by repeating the first half of Part 2 in the proof of Proposition \ref{PROP:Gibbsbilinear1},
we have 
$ \big|\mathcal{I}_{s_2, s_1, s_1, \theta}\big|
\lesssim \|f\|_{L^2_{n, \tau}}\|g\|_{ L^2_{n, \tau}}\|h\|_{L^1_n L^2_{\tau}},$
for all $h \in L^1_n L^2_\tau$.

On $B^c$, we have $|n|\sim|n_1|\sim|n_2|$.
Thus, $\frac{\jb{n}^{s_2+\frac{1}{2}}}{\jb{n_1}^{s_2} \jb{n_2}^{s_2}} \lesssim \jb{n}^{\frac{1}{2} -s_2} \lesssim 1$
for $s_2 \geq \frac{1}{2}$.
Hence, by repeating the second half of Part 2 in the proof of Proposition \ref{PROP:Gibbsbilinear1},
we have 
$ \big|\mathcal{I}_{s_2, s_2, s_2, \theta}\big|
\lesssim \|f\|_{L^\infty_n L^2_\tau}\|g\|_{L^\infty_n L^2_\tau} \|h\|_{L^1_n L^2_{\tau}},$
for all $h \in L^1_n L^2\tau$.
This proves the $X_\g^{s_2 -\frac{1}{2},\, \infty, \, b'}$ part of \eqref{Gibbsbilinear2}.
\end{proof}

\section{Construction of the Gibbs Measure}

In this section, we discuss the construction of the Gibbs measure $\mu$ for \eqref{SBO}
following Bourgain \cite{BO4}.
Once we construct the Gibbs measure, we can easily adapt the argument in Bourgain \cite{BO4, BOPARK}
to extend the local well-posedness result (Theorem \ref{THM:LWP})
to the global well-posedness almost surely on the statistical ensemble
and to establish the invariance of the Gibbs measure (Theorem \ref{THM:INVARIANCE}.)
The argument is standard and we omit the details.
Also, see Oh \cite{OH3} for the details on this part of the argument for the KdV systems.

 Given a Hamiltonian flow
\begin{equation*}
\begin{cases}
\dot{p}_i = \frac{\partial H}{\partial q_i} \\
\dot{q}_i = - \frac{\partial H}{\partial p_i} 
\end{cases}
\end{equation*}
on $\mathbb{R}^{2n}$ with Hamiltonian $ H = H(p_1, \cdots, p_n, q_1, \cdots, q_n)$,
Liouville's theorem states that the Lebesgue measure on $\mathbb{R}^{2n}$ is invariant under the flow.
From the conservation of the Hamiltonian $H$, the Gibbs measures
$e^{-\nu H} \prod_{i = 1}^{n} dp_i dq_i$ are also invariant,
where $\nu> 0$ is the reciprocal temperature.

In the context of NLS, Lebowitz-Rose-Speer \cite{LRS} considered the Gibbs measure of the form 
$d \mu = \exp (-\beta H(u)) \prod_{x\in \T} d u(x)$ where 
$H(u)$ is the Hamiltonian given by $H(u) = \frac{1}{2} \int |u_x|^2 \pm \frac{1}{p} \int |u|^p dx$.
In the focusing case (with  $-$), $H(u)$ is not bounded from below and this causes a problem.
Using the conservation of the $L^2$ norm,
they instead considered the Gibbs measure of the form 
$d \mu = \exp (-\beta H(u)) \chi_{\{\|u\|_{L^2} \leq B \}} \prod_{x\in \T} d u(x)$, i.e. with an $L^2$-cutoff.
This turned out to be a well-defined measure on $H^{\frac{1}{2}-}(\T) = \bigcap_{s<\frac{1}{2}} H^s(\T)$
(for $p < 6$ with any $B>0$, and $ p = 6$ with sufficiently small $B$.)
Bourgain \cite{BO4} continued this study and 
proved the invariance of $\mu$ under the flow of NLS
and the global well-posedness almost surely on the statistical ensemble.
Note that  \cite{BO4} appeared before the so-called Bourgain's method \cite{BO2} or the $I$-method \cite{CKSTT4},
i.e. there was virtually no method available to establish any GWP result 
from a LWP result whose regularity was between two conservation laws.
This was the case for NLS for $  4< p \leq 6$.
We use this idea to obtain a.s. GWP of the SBO system \eqref{SBO}.
The same idea was applied to show the invariance of the Gibbs measures
and a.s GWP for coupled KdV systems under some Diophantine conditions \cite{OH3}. 
Recently, Burq-Tzvetkov \cite{BT1} used 
similar ideas to prove a.s. GWP for 
the nonlinear wave equation on the unit ball in $\mathbb{R}^3$ under the radial symmetry.
Also, see other work by Tzvetkov related to this subject \cite{TZ1}, \cite{TZ2}.

Recall that the mean of $v$, the $L^2$ norm of $u$, 
and the Hamiltonian 
$H(u, v)  = \frac{1}{2}\int |u_x|^2 dx + \frac{\kk}{2}\int(D^\frac{1}{2}v)^2dx  + \frac{\al}{2} \int v |u|^2 dx$
are conserved under the flow of the SBO system \eqref{SBO}. 
In the following, we assume $\kk := - \frac{\al \g}{2\beta} > 0$
so that the quadratic part of the Hamiltonian \eqref{CONSERVED2} is nonnegative.
Note that $\frac{\al}{2} \int v|u|^2 dx$ is not sign-definite.
In particular,  $\exp \big( - \nu \frac{\al}{2} \int v|u|^2 dx\big)$
is not bounded from above as in the case of focusing NLS 
and KdV \cite{LRS}, \cite{BO4}.
This motivates us to define the Gibbs measure of the form
$d \mu = \chi_{\{ \|\phi\|_{L^2} \leq B\}} \chi_{\{|\ft{\psi}(0)| \leq B\}}
e^{ - \nu \frac{\al}{2} \int \psi|\phi|^2 dx}d\rho$,
where $d\rho $ is the Gaussian introduced in \eqref{GAUSS}.
Here, we associate $(\phi, \psi)$ with $(u(t), v(t))$ for fixed $t\in \R$. 
In particular, $\psi$ is real-valued.
For simplicity, we  set $\nu = 1$ for the rest of the paper.

Let $(a, b) = (a_n, b_n)_{n\in \mathbb{Z}}$ denote  the Fourier coefficients of $(\phi, \psi)$ on $\mathbb{T}$.
Since $\psi$ is real-valued, we have $b_{-n} = \cj{b_n}$. 
Let $B>0$ be a cutoff as above and consider the cylinder in $\mathbb{C}^{2N+1} \times \mathbb{R}^{2N+1}$ given by
\[ \Omega_{N, B} = \left\{  (a_n, b_n)_{ |n| \leq N } : \| a_n \|_{L^2_n}  \leq B \text{ and } |b_0|\leq B\right\} .\]

\noindent
Here, we abuse the notation and actually identify $\{b_n\}_{|n| \leq N}$ with
$(b_0, \text{Re}\,  b_1, \cdots, $ $\text{Re} \,b_N,\text{Im} \, b_1, \cdots, $ $\text{Im} \, b_N) \cong \mathbb{R}^{2N+1}$.
Also, define 
\[  \Omega_B = \left\{  (a_n, b_n)_{ n \in \mathbb{Z }} : \| a_n \|_{L^2_n}  \leq B \text{ and } |b_0|\leq B\right\}. \]

\noindent
Then, define the measure $\rho_N$ on $\mathbb{C}^{2N} \times \mathbb{R}^{2N} 
= \Big\{ (a_n, b_n)_{ \begin{smallmatrix}|n| \leq N \\ n \ne 0 \end{smallmatrix}} \Big\}$
with the normalized density
\[ d \rho_N = \wt{Z}_N^{-1} e^{-\frac{1}{2} \sum_{ |n| \leq N,  n \ne 0 } ( n^2 |a_n|^2+ \kk |n| |b_n|^2)} \prod_{ |n| \leq N,  n \ne 0 } d ( a_n \otimes b_n), \]

\noi
where
$\wt{Z}_N = \int_{\mathbb{C}^{2N} \times \mathbb{R}^{2N}} 
e^{-\frac{1}{2} \sum_{ |n| \leq N,  n \ne 0 } 
( n^2 |a_n|^2+ \kk |n| |b_n|^2)} \prod_{ |n| \leq N,  n \ne 0 } d ( a_n \otimes b_n). $
Note that this measure is the induced probability measure on $\mathbb{C}^{2N} \times \mathbb{R}^{2N}$ under the map
$ \omega \longmapsto \big\{ \big(n^{-1} f_n(\omega),  
\kk^{-\frac{1}{2}}|n|^{-\frac{1}{2}}g_n(\omega) \big); |n| \leq N, n\ne 0 \big\}$,
where $\{f_n(\omega)\}$ and $\{g_n(\omega)\}$ are i.i.d. standard complex  Gaussian random variables
(with $g_{-n} = \cj{g_n}$.)
Now, define the Gaussian measure $\rho$ on $\big\{ (a_n, b_n)_{  n \ne 0 } \big\}$
whose density is given by
\begin{equation} \label{gaussian}
 d \rho = \wt{Z}^{-1} 
e^{-\frac{1}{2} \sum_{  n \ne 0 } ( n^2 |a_n|^2+ \kk |n| |b_n|^2)} \prod_{  n \ne 0 } d ( a_n \otimes b_n), 
\end{equation}

\noi
where
$\wt{Z} = \int
e^{-\frac{1}{2} \sum_{   n \ne 0 } ( n^2 |a_n|^2+ \kk |n| |b_n|^2)} \prod_{  n \ne 0 } d ( a_n \otimes b_n). $

Recall the following definitions \cite{KUO}:
Given  a real separable Hilbert space $H$ with norm $\|\cdot \|$, 
let $\mathcal{F} $ denote the set of finite dimensional orthogonal projections $\mathbb{P}$ of $H$.
Then, define a cylinder set $E$ by  $E = \{ u \in H: \mathbb{P}u \in F\}$ where $\mathbb{P} \in \mathcal{F}$ 
and $F$ is a Borel subset of $\mathbb{P}H$,
and let $\mathcal{R} $ denote the collection of such cylinder sets.
Note that $\mathcal{R}$ is a field but not a $\s$-field.
Then, the Gauss measure $\rho$ on $H$ is defined 
by 
\[ \rho(E) = (2\pi)^{-\frac{n}{2}} \int_F e^{-\frac{\|u\|^2}{2}} du  \]

\noindent
for $E \in \mathcal{R}$, where
$n = \text{dim} \mathbb{P} H$ and  
$du$ is the Lebesgue measure on $\mathbb{P}H$.
It is known that $\rho$ is finitely additive but not countably additive in $\mathcal{R}$.

A seminorm $|||\cdot|||$ in $H$ is called measurable if for every $\eps>0$, 
there exists $\mathbb{P}_0 \in \mathcal{F}$ such that 
$ \rho( ||| \mathbb{P} u ||| > \eps  )< \eps$
for $\mathbb{P} \in \mathcal{F}$ orthogonal to $\mathbb{P}_0$.
Any measurable seminorm  is weaker  than the norm of $H$,
and $H$ is not complete with respect to $|||\cdot|||$ unless $H$ is finite dimensional.
Let $B$ be the completion of $H$ with respect to $|||\cdot|||$
and denote by $i$ the inclusion map of $H$ into $B$.
The triple $(i, H, B)$ is called an abstract Wiener space.

Now, regarding $v \in B^\ast$ as an element of $H^\ast \equiv H$ by restriction,
we embed $B^\ast $ in $H$.
Define, for a Borel set $F \subset \R^n$, 
\[ \wt{\rho} ( \{u \in B: ((u, v_1), \cdots, (u, v_n) )\in F\})
= \rho ( \{u \in H: (\jb{u, v_1}_H, \cdots, \jb{u, v_n}_H )\in F\}),\]

\noindent
where $v_j$'s are in $B^\ast$ and $(\cdot , \cdot )$ denote the natural pairing between $B$ and $B^\ast$.
Let $\mathcal{R}_B$ denote the collection of cylinder sets
$ \{u \in B: ((u, v_1), \cdots, (u, v_n) )\in F \}$
in $B$.

\begin{theorem}[Gross \cite{GROSS}]
$\wt{\rho}$ is countably additive in the $\s$-field generated by $\mathcal{R}_B$.
\end{theorem}

\noi
In the present context, let $H = \dot{H}^1 \times \dot{H}^\frac{1}{2}$ and  $B = \textup{H}^{s_1, s_2}$ 
with $0 < s_1 < \frac{1}{2} < s_2 < 1$. 
Then, we have 

\begin{proposition} \label{PROP:meas}
The seminorm  $\|\cdot\|_{B}$ is measurable.

\end{proposition}

\noindent
Hence, $(i, \dot{H}^1 \times \dot{H}^\frac{1}{2}, \text{H}^{s_1, s_2})$ is an abstract Wiener space, 
and $\rho$ in \eqref{GAUSS} and \eqref{gaussian} is countably additive in 
$\text{H}^{s_1, s_2}$ for
$0 < s_1  < \frac{1}{2} < s_2  < 1$.
Moreover, we have $\rho_N \rightharpoonup \rho$.  
(See \cite{Z}.)
For the proof of this proposition, see Oh \cite{OH3}.
Note that Bourgain used this norm 
(i.e.  $H = \dot{H}^1 $,  $B = H^{s_1, s_2}$ 
with $0 < s_1 < \frac{1}{2} < s_2 < 1$)
in studying the invariance of the Gibbs measure for mKdV \cite{BO4}.

Given an abstract Wiener space $(i, H, B)$, we have the following result due to Fernique \cite{FER}.
Also, see \cite[Theorem 3.1]{KUO}.

\begin{lemma} 
\label{LEM:FER}
There exists $ c > 0$ such that $ \int_B e^{c \|u\|_B^2} \rho(du) < \infty$.
Hence, there exists $ c' > 0$ such that $\rho ( \|u\|_B > K) \leq e^{-c'K^2}$.
\end{lemma}

Now, define $\Omega_{N, B} (s_1, s_2, K)$ and $\Omega_{B} (s_1, s_2, K)$ by
\begin{align*}
\Omega_{N, B} (s_1, s_2, K) & = \Big\{  (a_n, b_n)_{ |n| \leq N }  \in \Omega_{N, B}: 
\Big\| \sum_{|n| \leq N} (a_n, b_n) e^{i n x} \Big\|_{\text{H}^{s_1, s_2} }  \leq K \Big\}  \\
 \Omega_{ B} (s_1, s_2, K) & = \Big\{  (a_n, b_n)_{ n \in \mathbb{Z} }  \in \Omega_{ B}: 
\Big\| \sum_{ n \in \mathbb{Z} } (a_n, b_n) e^{i n x} \Big\|_{\text{H}^{s_1, s_2}}  \leq K \Big\}.  
\end{align*}

\noindent
Also, let $\wt{\Omega}_{N, B}$, $\wt{\Omega}_{B}$, $\wt{\Omega}_{N, B} (s_1, s_2, K)$,
  $\wt{\Omega}_{B} (s_1, s_2, K) $
be the restrictions of ${\Omega}_{N, B}$, ${\Omega}_{B}$, ${\Omega}_{N, B} (s_1, s_2, K)$,
  $\Omega_{B} (s_1, s_2, K) $ onto their mean 0 parts.
Then, basically from Lemma \ref{LEM:FER}, we have 
\begin{lemma} \label{tight0}
 Let $0 < s_1  < \frac{1}{2} < s_2  < 1$
For sufficiently large $K > 0$, we have
\[ \rho_N \big(\wt{\Omega}_{N, B} \setminus \wt{\Omega}_{N, B} (s_1, s_2, K) \big) \leq e^{-cK^2}, 
\text{ and } \ 
 \rho \big(\wt{\Omega}_{ B} \setminus  \wt{\Omega}_{ B} (s_1, s_2, K) \big) \leq e^{-cK^2},\]
where $c$ and the implicit constant are independent of $N$.
\end{lemma}

\noindent
The proof of this lemma is standard and is omitted.
See \cite{BO4, OH3}.

For the rest of the paper, let $0 < s_1 = \frac{1}{2} - < \frac{1}{2} < s_2 = 1- < 1$ with $s_2 < 2s_1$.
Let $\mathbb{P}_N$ be the projection onto the frequencies $|n| \leq N $ given by
$ \mathbb{P}_N \phi =\phi^N =  \sum_{|n| \leq N} a_n e^{ i nx}. $
Now, define the weighted Wiener measure $\mu_N$ on $\mathbb{R}^{2N+ 1} \times \mathbb{C}^{2N+ 1} 
= \big\{ (a_n, b_n)_{|n| \leq N } \big\}$ 
 by
\[ d \mu_N = Z_N^{-1} \exp \Big( -\frac{\al}{2} \int \mathbb{P}_N \psi \, |\mathbb{P}_N \phi |^2dx 
\Big) \chi_{\Omega_{N, B}} \ d (a_0, b_0) \otimes d \rho_N,  \]

\noindent
where
$ Z_N = \int_{\mathbb{C}^{2N+ 1} \times \mathbb{R}^{2N+ 1} }
\exp \left( -\frac{\al}{2} \int \mathbb{P}_N \psi \, |\mathbb{P}_N \phi |^2dx \right) \chi_{\Omega_{N, B}} \ d (a_0, b_0) \otimes d \rho_N. $
Also, define  the weighted Wiener measure $\mu$ on $\big\{ (a_n, b_n)_{  n \in \mathbb{Z} } \big\}$ by
\[ d \mu = Z^{-1} \exp \Big( -\frac{\al}{2} \int \psi  \, |\phi|^2dx \Big) \chi_{\Omega_B} 
\ d (a_0, b_0) \otimes d \rho,\]

\noindent
where
$Z = \int \exp \left( -\frac{\al}{2} \int \psi  \, |\phi|^2dx \right) \chi_{\Omega_B} \ d (a_0, b_0) \otimes d \rho. $
At this point, $Z$ need not be finite and thus $d \mu$ need not be a well-defined probability measure. 
Indeed, we have

\begin{lemma} \label{LEM:GABSCONTI}
For any $ r < \infty$, we have
\begin{align} \label{GABSCONTI1}
 \exp \left( -\frac{\al}{2} \int \mathbb{P}_N \psi \, |\mathbb{P}_N \phi |^2dx \right) \chi_{\Omega_{N, B}}
& \in L^r ( d (a_0, b_0) \otimes d \rho_N ) 
\\ \label{GABSCONTI2}
\exp \left( -\frac{\al}{2} \int \psi  \, |\phi|^2dx \right) \chi_{\Omega_B} & \in L^r (d (a_0, b_0) \otimes d \rho).
\end{align}
\end{lemma}

\noindent
In particular, $d\mu$ is a well-defined probability measure.
Moreover, we have  $d\mu_N \ll d(a_0, b_0) \otimes d\rho_N$ and   $d\mu \ll d(a_0, b_0) \otimes d\rho$.
Then, from Lemma \ref{tight0}, we have 
\begin{corollary} [tightness of $\mu_N$ and $\mu$] \label{tight}
For large $K > 0$, we have
\[ \mu_N \big(\Omega_{N, B} \setminus \Omega_{N, B} (s_1, s_2, K) \big) \leq e^{-cK^2}, 
\text{ and } \
 \mu \big(\Omega_{ B} \setminus  \Omega_{ B} (s_1, s_2, K) \big) \leq e^{-cK^2},\]
where $c$ and the implicit constant are independent of $N$.
\end{corollary}

\noindent
The proof of Lemma \ref{LEM:GABSCONTI} is based on Bourgain's argument in \cite{BO4}.
We have an additional difficulty since $\psi \in H^{0-} \setminus L^2$ almost surely.
i.e. the argument in in the sub-$L^2$ setting and we need to employ a probabilistic argument.

\begin{proof} [Proof of Lemma \ref{LEM:GABSCONTI}]
We  prove only \eqref{GABSCONTI2}.
First, note that on $ \Omega_B $, we have $|a_0|, |b_0| \leq B$.
Then, from Young's inequality with $p = 2-$ and $\frac{1}{p}+\frac{1}{p'} = 1$, we have
\begin{align*}
 \int \psi   |\phi|^2 dx  
& \leq \big\| \ft{|\phi|^2} \big\|_{L_n^{p}}\big\| \ft{\psi} \big\|_{L_n^{p'}} 
 \leq \tfrac{1}{p'}\big\| \ft{|\phi|^2} \big\|^{p'}_{L_n^{p}}
+ \tfrac{1}{p}\| \ft{\psi} \|^{p}_{L_n^{p'}} \\
& \leq \tfrac{1}{p'}\| \ft{\phi} \|^{2p'}_{L_n^{\frac{2p}{1+p}}}
+ \tfrac{1}{p}\| \ft{\psi} \|^{p}_{L_n^{p'}} 
  \lesssim |a_0|^{2p'} +  \big\| \{ \ft{\phi}\}_{n \ne 0} \big\|^{2p'}_{L_n^{\frac{2p}{1+p}}} 
+  |b_0|^p + \big\| \{ \ft{\psi} \}_{n \ne 0}\big\|^{p}_{L_n^{p'}} .
\end{align*}

\noindent
Since $\frac{2p}{1+p} = \frac{4}{3}-$, we have 
\begin{align*}
\Big\| & e^ {-\frac{\al}{2}\int_{\mathbb{T}} \psi|\phi|^2 dx}  \chi_{\Omega_B} \Big\|_{L^r (d (a_0, b_0) \otimes d\rho)} \\
& \lesssim \bigg(\intt_{\substack{|a_0| \leq B \\ |b_0|\leq B }} 
e^ { c(B^p + B^{2p'})} \chi_{\Omega_B}d (a_0, b_0) \bigg)^\frac{1}{r}
\Big\|e^{c (\| \{ \ft{\phi}\}_{n \ne 0} \|^{4+}_{L_n^{\frac{4}{3}-}}
  + \| \{ \ft{\psi} \}_{n \ne 0}\|^{2-}_{L_n^{2+}} )} \chi_{\Omega_B}\Big\|_{L^r(d\rho)} \\
& \lesssim C_B  
\Big\|\exp\big({c( \| \{ \ft{\phi}\}_{n \ne 0} \|^{4+}_{L_n^{\frac{4}{3}-}}
  + \| \{ \ft{\psi} \}_{n \ne 0}\|^{2-}_{L_n^{2+}} )} \big) \chi_{\Omega_B}\Big\|_{L^r(d\rho)}. 
\end{align*}

\noindent
Let $d \rho_1= d\rho\big|_{(a_n)_{n\ne 0}}$ and $d\rho_2 = d\rho\big|_{(b_n)_{n\ne 0}}$, 
i.e.
\[ d \rho_1 = \wt{Z}_1^{-1} e^{-\frac{1}{2} \sum_{n\ne 0} n^2 |a_n|^2} \prod_{n \ne 0 } d a_n,
\ \ \text{ and } \ \ 
d \rho_2 = \wt{Z}_2^{-1} e^{-\frac{1}{2} \sum_{n\ne 0} \kk |n| |b_n|^2} \prod_{n \ne 0 } d b_n, 
\]

\noindent
where
$\wt{Z}_1 = \int e^{-\frac{1}{2} \sum_{n\ne 0} n^2 |a_n|^2} \prod_{n \ne 0 } d a_n$
and 
$\wt{Z}_2 = \int e^{-\frac{1}{2} \sum_{n\ne 0} \kk |n| |b_n|^2} \prod_{n \ne 0 } d b_n.$
Then, since $d\rho = d\rho_1 \otimes d\rho_2$,
it suffices to prove, for arbitrary $r < \infty$,
\begin{align} \label{GLR1}
\exp \big({ \|\ft{\phi} \|_{L_n^{\frac{4}{3}-}}^{4+}  } \big) \chi_{ \{\|\phi\|_{L^2} \leq B\} } 
& \in L^r (d\rho_1), 
\\ \label{GLR2}
\exp \big( { \|\ft{\psi}\|_{L_n^{2+}}^{2-}  } \big) & \in L^r (d\rho_2), 
\end{align}

\noindent
for  $\phi$ and $\psi$ with mean 0.
First, note that,  
by H\"older inequality, we have
\[\big\|\{\ft{\phi}(n)\}_{n \sim M} \big\|_{L_n^{\frac{4}{3}-}}^{4+} 
= \Big( \sum_{|n|\sim M} |a_n|^{\frac{4}{3}-}\Big)^{\frac{3}{4}+}
\leq C  M^{\frac{1}{4}+} \Big( \sum_{|n|\sim M} |a_n|^2\Big)^\frac{1}{2}.\]

\noindent
for any $M$ dyadic.
Then, we see that the proof of \eqref{GLR1} is basically the same as that of  
$\exp \big({ \|\phi \|_{L_x^{4+}}^{4+}  } \big) \chi_{ \{\|\phi\|_{L^2} \leq B\} } 
\in L^r (d\rho_1)$ in  Bourgain \cite{BO4}.

Now, we turn to the proof of \eqref{GLR2}.
First, assume that 
$(i, \dot{H}^\frac{1}{2}, \ft{L}_n^{2+})$ is an abstract Wiener space (with respect to $\rho_2$), 
where $\ft{L}_n^{2+}$ is the space defined via the norm
$\|\psi\|_{\ft{L}_n^{2+}} =\|\ft{\psi}\|_{{L}_n^{2+}}.$
Then, by Lemma \ref{LEM:FER}, we have
\begin{align*}
& \big\|  e^{ \| \psi\|_{\ft{L}_n^{2+}}^{2-}  }   \big\|^r_{L^r(d\rho_2)} 
\leq  \int_{\{ \|\psi\|_{\ft{L}_n^{2+}} \leq K\}} e^{ r\|\psi\|_{\ft{L}_n^{2+}}^{2-}  } d \rho_2
+ \sum_{j = 0}^\infty \int_{\{ 2^j K \leq \|\psi\|_{\ft{L}_n^{2+}} < 2^{j + 1} K\}} e^{ r\|\psi\|_{\ft{L}_n^{2+}}^{2-}  } d \rho_2 \\
&\hphantom{X} \leq e^{ r K^{2-}}
+ \sum_{j = 0}^\infty e^{ r (2^{j + 1} K)^{2-}}  \rho_2 \big[ \|\psi\|_{\ft{L}_n^{2+}} \geq 2^j K \big] 
\lesssim e^{ r K^{2-}} +  \sum_{j = 1}^\infty e^{-c(2^{j}K)^2 + 2^{2-}r (2^{j}K)^{2-}} 
< \infty. 
\end{align*}

Hence, it remains to show that $(i, \dot{H}^\frac{1}{2}, \ft{L}_n^{2+})$ is an abstract Wiener space. 
Recall that $\psi = \sum_{n \ne 0} b_n e^{inx} 
= \sum_{n \ne 0} \frac{g_n(\omega)}{|\kk|^\frac{1}{2}|n|^\frac{1}{2}}  e^{inx}$, 
where $\{g_n(\omega)\}$ is a sequence of i.i.d. complex Gaussian random variables
with $g_{-n} = \cj{g_n}$.
Note that it suffices to show that for given $\eps > 0$
there exists large $M_0$ such that $\rho_2 \big[ \|\mathbb{P}_{> M_0} \psi\|_{\ft{L}_n^{2+}} > \eps \big] < \eps$,
 where $\mathbb{P}_{> M_0}$ is the Dirichlet projection onto the frequencies $|n| > M_0$.
First, we present a lemma which provides an a.s.  decay at high frequencies in the sub-$L^2_n$ setting.

\begin{lemma} \label{LEM:decay}
Let $\{g_n\}$ be a sequence of i.i.d Gaussian random variables.
Then, for $M$ dyadic and $\dl >0$, we have
\begin{equation} \label{EQ:decay}
\lim_{M\to \infty} M^{1-\dl} \frac{\max_{|n|\sim M } |g_n|^2}{ \sum_{|n|\sim M} |g_n|^2} 
 = 0, \text{ a.s.} 
\end{equation}
\end{lemma}

\begin{proof}
We show $ \lim_{n \to \infty} n^{1-\dl} \frac{\max_{1\leq j \leq n } |g_j|^2}{ \sum_{j = 1}^n |g_j|^2} = 0$, 
a.s.
Let $X_n = |g_n|^2$.  Then, $\{X_n\}_{n \in \mathbb{N}} $ is a sequence of i.i.d random variables with $E|X_n|  = 1< \infty$.
By Kolmogorov's SLLN, we have $\frac{S_n}{n} \to 1$ a.s., where $S_n = \sum_{j = 1}^n X_j $.
Now, fix $\eps> 0$ and $ \theta > 0$.
By Chebyshev's inequality, we have 
\begin{equation} \label{chebyshev}
n^{2+\theta} \mathbb{P} [ X_1 > \eps n^\delta ] 
\leq \eps^{-\frac{2+ \theta}{\dl}} \int_{\{\omega: X_1(\omega) > \eps n^\dl\} }
X_1^{\frac{2+ \theta}{\dl}} \mathbb{P}(d \omega) 
\leq \eps^{-\frac{2+ \theta}{\dl}} \mathbb{E}[X_1^{\frac{2+ \theta}{\dl}}] < \infty
\end{equation}

\noi
for all $n \in \mathbb{N}$ as long as $\delta > 0$.
Let $M_n = \max_{1\leq j \leq n } X_j$. 
Then, by the independence of $X_j$ and \eqref{chebyshev}, we have 
\begin{align*} 
 \sum_{n = 1}^\infty \mathbb{P}[ & M_n >  \eps n^{\dl}] 
\leq \sum_{n = 1}^\infty n \mathbb{P}[ X_1 >  \eps n^{\dl}] 
= \sum_{n = 1}^\infty n^{-1-\theta} n^{2+\theta} \mathbb{P}[ X_1 >  \eps n^{\dl}] 
\lesssim  \sum_{n = 1}^\infty n^{-1-\theta} < \infty.
\end{align*}

\noi
Then, by Borel-Cantelli lemma,
$\mathbb{P} [ \frac{M_n}{n^{\dl}} > \eps,  \text{ i.o.} ] = 0$. 
This implies that $\limsup_{n \to \infty} \frac{M_n}{n^{\dl}} \leq \eps$, a.s. 
Since $\limsup_{n \to \infty} \frac{M_n}{n^{\dl}}$ is a tail function, 
it is a.s. constant. 
Noting that this nonnegative constant is bounded above by any $\eps > 0$, we conclude that 
$\limsup_{n \to \infty} \frac{M_n}{n^{\dl}} = 0$, a.s.
\end{proof}

Now, fix $\eps > 0$ and $\dl \in ( 0, \frac{1}{2})$.
Then, by Lemma \ref{LEM:decay} and Egoroff's theorem, 
there exists a set $E$ with $\rho_2 (E^c) < \frac{1}{2} \eps$
such that the convergence in Lemma \ref{LEM:decay} is uniform on $E$.
i.e. we can choose dyadic $M_0$ large such that 
\begin{align} \label{A:decay}
\frac{\| \{g_n (\omega) \}_{|n| \sim M} \|_{L_n^{\infty}} }
{\| \{g_n (\omega) \}_{|n| \sim M} \|_{L_n^{2}} }
  \leq M^{-\dl}, 
\end{align}

\noindent 
for  all $\omega \in E$  and dyadic $M > M_0$.
In the following, we will work only on $E$ and  drop `$\cap E$' for notational simplicity. 
However, it should be understood that all the events are under the intersection with $E$ and 
thus \eqref{A:decay}  holds.

Let $\{\s_j \}_{j \geq 1}$ be a sequence of positive numbers such that $\sum \s_j = 1$,
and let $ M_j = M_0 2^j$ dyadic.
Note that $\s_j = C 2^{-\ld j} =C M_0^\ld M_j^{-\ld}$ for some small $\ld > 0$
(to be determined later.)
Then,  we have
\begin{equation} \label{A:subadd} 
\rho_2 \big[ \big\| \mathbb{P}_{>M_0} \psi \big\|_{\ft{L}_n^{2+}} > \eps \big] 
\leq \sum_{j = 0}^\infty \rho_2 \big[ \| \{b_n \}_{|n| \sim M_j} \|_{L_n^{2+}}  > \s_j \eps \big].
\end{equation}

\noindent
By H\"older inequality and \eqref{A:decay}, we have  
\begin{align*}
\| \{ & b_n \}_{|n| \sim M_j} \|_{L_n^{2+\theta}} 
\sim M_j^{-\frac{1}{2}} \| \{ g_n \}_{|n| \sim M_j} \|_{L_n^{2+\theta}} 
\leq  M_j^{-\frac{1}{2}} \| \{ g_n \}_{|n| \sim M_j} \|_{L_n^{2}}^\frac{2}{2+\theta}
  \| \{ g_n \}_{|n| \sim M_j} \|_{L_n^{\infty}}^\frac{\theta}{2+\theta} \\
& \leq  M_j^{-\frac{1}{2}} \| \{ g_n \}_{|n| \sim M} \|_{L_n^{2}}
\Bigg(\frac{  \| \{ g_n \}_{|n| \sim M_j} \|_{L_n^{\infty}} }
{\| \{ g_n \}_{|n| \sim M_j} \|_{L_n^{2}}} \Bigg)^\frac{\theta}{2+\theta}
\leq  M_j^{-\frac{1}{2} -\dl \frac{\theta}{2+\theta}} \| \{ g_n \}_{|n| \sim M_j} \|_{L_n^{2}}
\end{align*}

\noindent
a. s. where $\theta = 0+$.
Thus, if we have $\|\{b_n \}_{|n| \sim M_j} \|_{L_n^{2+}}  > \s_j \eps$,
 then we have
 $\| \{ g_n \}_{|n| \sim M_j} \|_{L_n^{2}} 
 \gtrsim R_j $
 where $R_j := \s_j  \eps M_j^{\frac{1}{2}+\dl \frac{\theta}{2+\theta}} $.
Now, take $\ld$ sufficiently small such that $\dl \frac{\theta}{2+\theta}-\ld > 0$.
By a direct computation in the polar coordinates with 
$R_j = \s_j \eps M_j^{\frac{1}{2}+\dl \frac{\theta}{2+\theta}} 
= C \eps M_0^\ld M_j^{\frac{1}{2}+\dl \frac{\theta}{2+\theta}-\ld}
= C \eps M_0^{\ld} M_j^{\frac{1}{2}+} $, 
we have 
\begin{align} \label{GLR6}
\rho_2\big[  \| \{ g_n \}_{|n| \sim M_j} \|_{L_n^{2}}  \gtrsim R_j \big]
 \sim \int_{B^c(0, R_j)} e^{-\frac{1}{2}|g|^2} \prod_{|n| \sim M_j} dg_n
 \lesssim \int_{R_j}^\infty e^{-\frac{1}{2}r^2}  r^{2 \cdot \# \{|n| \sim M_j\} -1} dr,
\end{align}

\noindent 
Note that the implicit constant in the inequality is 
$\s(S^{2 \cdot \# \{|n| \sim M_j\} -1})$, a surface measure of the 
$2 \cdot \# \{|n| \sim M_j\} -1$ dimensional unit sphere.
We drop it since $\s(S^n) = 2 \pi^\frac{n}{2} / \Gamma (\frac{n}{2}) \lesssim 1$.
By change of variable $t = M_j^{-\frac{1}{2}}r$, we have
$r^{2 \cdot \# \{|n| \sim M_j\} -2} \lesssim r^{4M_j} \sim M_j^{2M_j} t^{4M_j}.$
Since $t > M_j^{-\frac{1}{2}} R_j = C \eps M_0^{\ld} M_j^{0+}$, we have
$M_j^{2M_j} = e^{2M_j \ln M_j} < e^{\frac{1}{8}M_jt^2}$ 
and $ t^{4M_j} < e^{\frac{1}{8}M_jt^2}$
for $M_0$ sufficiently large.
Thus, we have
$r^{2 \cdot \# \{|n| \sim M_j\} -2} < e^{\frac{1}{4}M_jt^2} = e^{\frac{1}{4}r^2}$ 
for $ r > R.$
Hence,  we have
\begin{align} \label{A:highfreq1}
\rho_2\big[   \| \{ g_n & \}_{|n| \sim M_j} \|_{L_n^{2}}  \gtrsim R_j \big]
 \leq C \int_{R_j}^\infty e^{-\frac{1}{4}r^2} r dr  \\
& \leq e^{-cR_j^2} = e^{-cC^2 M_0^{2\ld} M_j^{1+} \eps^2}
=  e^{-cC^2 M_0^{1+2\ld+} 2^{j+} \eps^2}. \notag
\end{align}

\noindent
From \eqref{A:subadd} and \eqref{A:highfreq1}, we have
\begin{align*} 
\rho_2  [ \| \{b_n  \}_{|n| > M_0} \|_{L_n^{2+}} > \tfrac{1}{2}\eps ]
\leq \sum_{j =1}^\infty 
\rho_2 \big[ \| \{ g_n \}_{|n| \sim M_j} \|_{L_n^{2}}  > R_j \big]
 \leq e^{-c' M_0^{1+2\ld+} 2^{j+} \eps^2} <\tfrac{1}{2}\eps, 
\end{align*}

\noi
by choosing $M_0$ sufficiently large.
 This completes the proof of Lemma \ref{LEM:GABSCONTI}.
\end{proof}

\section{Appendix: On the Ill-posedness Results in $H^{s}(\T) \times H^{s-\frac{1}{2}}(\T)$ for $s < \frac{1}{2}$. }

There are  so called ``ill-posedness" results for dispersive equations such as NLS and KdV.
However, this term often refers to the necessary conditions for uniform continuity or smoothness of 
the solution map $\Phi^t : u_0 \in H^s \longmapsto u(t) \in H^s$. 
In such cases, the Cauchy problem is not necessarily ill-posed in the sense of the usual definition, even when these results hold.
However, since the contraction argument provides analytic dependence on the initial data, 
it is often natural to consider a strengthened  notion of 
well-posedness requiring the solution map to be uniformly continuous/smooth.
In this latter sense, the following results may be regarded as ``ill-posedness" results. 

Here, we follow Bourgain's argument in \cite{BO3}. 
Consider the following Cauchy problem:
\begin{equation} \label{illposedSBO}
\begin{cases}
i u_t + u_{xx} = \al  vu \\
v_t + \g \mathcal{H} v_{xx} = \beta (|u|^2)_x \\
\big( u(x, 0), v(x, 0) \big) = \big( \delta \phi(x) ,  \delta \psi(x) \big) \in H^s(\T) \times H^{s-\frac{1}{2}} (\T)
\end{cases}
\end{equation}

\noindent
where $\dl \geq 0$.
Let $\big( u(x, t;\dl), v(x, t;\dl) \big)$ or $\big( u(t;\dl), v(t;\dl) \big)$ denote the solution to \eqref{illposedSBO} .
First, note that with $\dl = 0$, $\big( u(x, t; 0), v(x, t; 0) \big)  \equiv 0$ is the unique solution.
Also, by writing as integral equations, the solution $\big(u(t;\dl), v(t;\dl) \big)$ to \eqref{illposedSBO}  can be written as
\begin{equation*}
\begin{cases}
u(t;\dl) = \dl U(t) \phi +i \al  \int_0^t U(t - t') uv (t') dt' \\
v(t;\dl) = \dl V(t) \psi - \beta \int_0^t V(t - t') \dx \big( |u (t')|^2\big)  dt' .
\end{cases}
\end{equation*}

\noindent 
where $U(t) = e^{it\dx^2}$ and $V(t) = e^{-\g \mathcal{H} \dx^2}$.
By taking derivatives in $\dl$ at $\dl = 0$,  we have 
$\dd u(t;0) = U(t) \phi =: \phi_1$ and $\dd v(t;0) = V(t) \psi =: \psi_1$.
By taking the 2nd and 3rd derivatives in $\dl$ at $\dl = 0$, we have
\begin{equation*}
\begin{cases}
\dd^2 u(t; 0) = 2 i\al  \int_0^t U(t - t')  \big( \phi_1 \psi_1 \big) (t') dt'  =: \phi_2 \\
\dd^2 v(t; 0) = - 2 \beta \int_0^t V(t - t') \dx \big(  |\phi_1|^2  \big) (t') dt'  =: \psi_2.
\end{cases}
\end{equation*}

\noindent
Then, it follows that if the solution map $\Phi^t: (u_0, v_0)  \in H^s \times H^{s-\frac{1}{2}} 
\longmapsto \big( u(t), v(t) \big) \in H^s \times H^{s-\frac{1}{2}}$ is $C^2$ for fixed $t \ll 1$, then we must have
\begin{equation} \label{CK}
\left\| \dd^2 \big(u, v\big) (\cdot, t; 0)   \right\|_{H^s_x \times H^{s-\frac{1}{2}}_x} 
= \| (\phi_2, \psi_2) (\cdot, t) \|_{H^s_x \times H^{s-\frac{1}{2}}_x}
\lesssim \left\| (\phi, \psi ) \right\|^2_{H^s\times H^{s-\frac{1}{2}}} 
\end{equation}

\noindent
from the smoothness of $\Phi^t$ at the zero solution. 
In the following, we present the proof of Theorem \ref{THM:illposed}, 
assuming that the SBO system \eqref{SBO} is locally well-posed in $H^s (\T) \times H^{s-\frac{1}{2}}(\T)$ 
over a small time interval and fix $t \ll 1$ such that the solution map $\Phi^t$ is well-defined.

\begin{proof}[Proof of Theorems \ref{THM:illposed}]

Recall that
with $n = n_1 +n_2$, we have
\[  Q (n, n_1) :=  \g |n| n -  n_1^2+ n_2^2 = n \big( (1 + \g \text{sgn}(n) ) n - 2n_1). \]

\noindent
Let $c_\g = c_\g (n) = \frac{1 + \g \text{sgn}(n)}{2}$ and  $d_\g = 1 - c_\g$.
Thus, $Q(n, n_1) = 0$ when $n_1 = c_\g n$ and $n_2 = d_\g n$. 
Now, let $\| \rho \|$ denote the closest integer to $\rho$.
(If $\rho - [\rho] = \frac{1}{2}$, let $\|\rho\| = [\rho]$, where $[\, \cdot\, ]$ is the integer part function.)

Given $N \in \mathbb{N}$,  let $\psi \equiv 0$ and  
$\phi(x) = N^{-s} \big( e^{i \, \|c_\g N\|x} + e^{-i \, \|d_\g N\|x} \big)$ with $c_\g = c_\g(N)$ and $d_\g = d_\g(N)$.
Then, we have $\| (\phi, \psi) \|_{H^s \times H^{s - \frac{1}{2}}}  = \| \phi \|_{H^s} \sim 1$.
A direct computation shows that $\psi_1 = \phi_2 \equiv 0$ and 
\begin{equation*} 
 \phi_1 (x, t)  = N^{-s} \big( e^{i  \|c_\g N\|x - i\|c_\g N\|^2t} + e^{-i  \|d_\g N\|x - i\|d_\g N\|^2t} \big).
\end{equation*}

\noindent
Using $ \|c_\g N\| + \|d_\g N\| = N$, we have 
\begin{equation*} \label{Vpsi2}
V(t - t') \dx |\phi_1 (x, t')|^2 = - 2 N^{-2s + 1} \sin \big( N x - \g |N|N t+ Q(N, \|c_\g N\|) t' \big).
\end{equation*}

Suppose $\g \in \mathbb{Q}$ with $|\g| \ne 1$. 
Then, for $n = N \in \mathbb{N}$, we have $c_\g (n) = \frac{1+\g}{2}\in \mathbb{Q}$.
Thus, there exist infinitely many $N \in \mathbb{N}$ such that $c_\g N, \, d_\g N \in \mathbb{Z}$.
Hence, we have $Q(N, \|c_\g N\|) = Q(N, c_\g N) = 0$ for all such $N$.
Then, 
\[\psi_2 = 4 \beta N^{-2s + 1}t \sin \big( N x - \g |N|N t)\]

\noindent
and thus $\| (\phi_2, \psi_2) \|_{H^s\times H^{s - \frac{1}{2}}} \sim \|\psi_2\|_{H^{s-\frac{1}{2}}} \sim N^{-s + \frac{1}{2}}$.
In view of \eqref{CK}, by letting $N \to \infty$, this implies $s \geq \frac{1}{2}$ if the solution map $\Phi^t$ is $C^2$.

Now, suppose $\g \notin \mathbb{Q}$.
Then, we have  $Q(N, \|c_\g N\|) \ne 0$ for any $n \in \mathbb{N}$.
From \eqref{Vpsi2}, we have 
\begin{align*} 
\psi_2 (x, t) 
& =  4 \beta N^{-2s+1} \frac{ \cos ( Nx -\g|N|N t) - \cos ( Nx  -(\|c_\g N\|^2 -\|d_\g N\|^2)t) }{Q(N, \|c_\g N\|)} \\
& = 8 \beta N^{-2 s + 1} \frac{ \sin \big(Nx - 2\g |N| N t + Q(N, \|c_\g N\|) t\big) \sin \big(Q(N, \|c_\g N\|) t\big)}{Q(N, \|c_\g N\|)}.
\end{align*} 

\noindent
Recall that Dirichlet Theorem \cite{LANG} says that for given $\rho \in \R \setminus \mathbb{Q}$, 
there exist infinitely many $(p, q) \in \mathbb{Z}^2$ such that
$ \big| \rho - \frac{p}{q} \big| < \frac{1}{q^2}.$
In our context, it says that there exist infinitely many $N \in \mathbb{N}$ such that
\[|Q(N, \|c_\g N\|)| = \min_{n_1\in \mathbb{Z} } |Q( N, n_1)| = N^2 \Big| c_\g - \frac{2 n_1}{N} \Big| \leq 1.\]

\noindent
Then, for such $N$, we have 
$\| (\phi_2, \psi_2) \|_{H^s\times H^{s - \frac{1}{2}}} \sim N^{-s + \frac{1}{2}}$.
In view of \eqref{CK}, by letting $N \to \infty$, this implies $s \geq \frac{1}{2}$ if the solution map $\Phi^t$ is $C^2$.

Finally, we will  consider the case when $|\g| = 1$.
Without loss of generality, assume $\g = 1$.
Then, $Q(N, N) = 0 $ for all $N \in \mathbb{N}$, i.e. $c_\g = 1$ and $d_\g = 0$.
(When $\g = -1$, we have $Q(N, N) = 0 $ for all $N \in \mathbb{Z}_{<0}$
and the following argument can be easily modified.)
Given $N \in \mathbb{N}$,  let $\psi \equiv 0$ and  
$\phi(x) = N^{-s}  e^{i Nx} +1$.
Then, we have $\| (\phi, \psi) \|_{H^s \times H^{s - \frac{1}{2}}}  = \| \phi \|_{H^s} \sim 1$.
A direct computation shows that $\psi_1 = \phi_2 \equiv 0$ and 
$ \phi_1 (x, t)  = N^{-s} e^{iNx - iN^2t} + 1$.
Thus, we have 
$ |\phi_1(x, t)|^2 = 2 N^{-s} \cos (N x - N^2 t) + N^{-2} + 1$ in this case
and 
\begin{equation*} 
V(t - t') \dx |\phi_1 (x, t')|^2 = - 2 N^{-s + 1} \sin \big( N x - N^2 t \big).
\end{equation*}

\noindent
since $Q(N, N) = 0$.
Therefore, we have 
\[\psi_2 = 4 \beta N^{-s + 1}t \sin \big( N x - N^2 t)\]

\noindent
and thus $\| (\phi_2, \psi_2) \|_{H^s\times H^{s - \frac{1}{2}}} \sim \|\psi_2\|_{H^{s-\frac{1}{2}}} \sim N^{\frac{1}{2}}$.
In view of \eqref{CK}, by letting $N \to \infty$, this implies that the solution map $\Phi^t$ can {\it never} be  $C^2$ for any $ s\in \R$.
\end{proof}

\smallskip

\noindent
{\bf Acknowledgements:} 
The author would like to thank Prof. Luc Rey-Bellet for mentioning the work of Gross \cite{GROSS}
and Kuo \cite{KUO}.

\end{document}